\documentclass[12pt]{article}
\usepackage{amsfonts}
\usepackage{amssymb, amscd, amsthm}
\usepackage{latexsym}
\usepackage{amsmath}
\usepackage[all]{xy}
\usepackage[dvips]{graphicx}
\usepackage{mathrsfs}
\usepackage{fancyhdr}

\newtheorem{lemma}[subsubsection]{Lemma}
\newtheorem{thm}[subsubsection]{Theorem}
\newtheorem{prop}[subsubsection]{Proposition}
\newtheorem{rem}[subsubsection]{Remark}
\newtheorem{coro}[subsubsection]{Corollary}

\newtheorem{example}[subsubsection]{Example}

\newcommand{\ra}{\rightarrow}
\newcommand{\mhu}{M_X^H(u)}
\newcommand{\mhc}{M_X^H(c)}
\newcommand{\lu}{\lambda_u}
\newcommand{\lc}{\lambda_c}
\newcommand{\lcn}{\lambda_{c^r_n}}
\newcommand{\mo}{\mathcal{O}}
\newcommand{\mf}{\mathcal{F}}
\newcommand{\mg}{\mathcal{G}}
\newcommand{\ma}{\mathcal{A}}
\newcommand{\mb}{\mathcal{B}}
\newcommand{\me}{\mathcal{E}}

\newcommand{\mc}{\mathcal{C}}
\newcommand{\mk}{\mathcal{K}}
\newcommand{\z}{\Theta}
\newcommand{\hk}{\textbf{K}}
\newcommand{\ha}{\textbf{A}}
\newcommand{\omu}{\Omega_u}

\newcommand{\ls}{|L|}
\newcommand{\pone}{\mathbb{P}^1}

\begin{document}
\fontsize{12pt}{14pt} \textwidth=14cm \textheight=21 cm
\numberwithin{equation}{section}

\title{Determinant line bundles on Moduli spaces of pure sheaves on rational surfaces and Strange Duality}
\author{Yao YUAN}
\date{\small\textsc SISSA, Via Bonomea 265, 34136, Trieste, ITALY
\\ yuanyao@sissa.it}
\maketitle
\begin{flushleft}{\textbf{Abstract:}}
Let $\mhu$ be the moduli space of semi-stable pure sheaves of class $u$ on a smooth complex projective surface $X$.  We specify $u=(0,L,\chi(u)=0),$  i.e. sheaves in $u$ are of dimension $1$.   
There is a natural morphism $\pi$ from the moduli space $\mhu$ to the linear system $\ls$.
We study a series of determinant line bundles $\lcn$ on $\mhu$ via $\pi.$ 
Denote $g_L$ the arithmetic genus of curves in $\ls.$
For any $X$ and $g_L\leq0$,  we compute the generating function $Z^r(t)=\sum_{n}h^0(\mhu,\lcn)t^n$.   For $X$ being $\mathbb{P}^2$ or $\mathbb{P}(\mo_{\pone}\oplus\mo_{\pone}(-e))$ with $e=0,1$,  we compute $Z^1(t)$ for $g_L>0$ and $Z^r(t)$ for all $r$ and $g_L=1,2$.  Our results provide a numerical check to Strange Duality in these specified situations,  together with G\"ottsche's computation.  And in addition,  we get an interesting corollary (Corollary \ref{trsame}) in the theory of compactified Jacobian of integral curves.  \end{flushleft}
\section{Introduction}
It is an interesting problem in its own right to determine the generating function of the sections of determinant line bundles on moduli space of pure sheaves of dimension 1.  But an additional motivation of our work comes from the so-called Strange Duality conjecture due to Le Potier. 

Let $X$ be any projective scheme of dimension $d$ over an algebraically closed field of characteristic zero,  and $H$ the very ample divisor on it.  Let $\hk(X)$ be the Grothendieck group of $X$.  For any two elements $u,c\in\hk(X)$,  we say they are orthogonal to each other,  i.e. $u\perp c,$  if $\chi(u\otimes c)=0$ with $\chi$ being the holomorphic Euler characteristic.   For any $u\in \hk(X)$,  there is a projective scheme $\mhu$ corepresenting the functor of semi-stable sheaves with respect to $H$ in $u$.  Let $h=[\mo_H].$  Given any other element $c\in\hk(X)$ such that 
\begin{equation}\label{bot}c\in u^{\bot}\cap (h,h^2,\ldots,h^d)^{\bot\bot},\end{equation}
 we then can associate to $c$ a well-defined determinant line bundle $\lc.$  

If $X$ is a smooth complex projective curve,  then given any two elements $c,u\in\hk(X)$ satisfying (\ref{bot}),  the locus in $\mhu\times\mhc$ 
\begin{equation}\label{divisor}D_{\lambda}=\{(E,F)\in\mhu\times\mhc~s.t.~h^0(E\otimes F)\neq0\}\end{equation}
is a divisor of the line bundle $\lu\boxtimes\lc$ on $\mhu\times\mhc$ and induces a morphism $D$ well-defined up to scalars,
\begin{equation}\label{sdm}D:H^0(\mhu,\lc)^{\vee}\ra H^0(\mhc,\lu).\end{equation}
The Strange Duality conjecture for curves due to Beauville and Donagi-Tu says that the map $D$ in (\ref{sdm}) is an isomorphism.  This conjecture has been studied by many people and has recently been proven.  (For generic curves:  Prakash Belkale in 2006;  for all curves:  Alina Marian,  Dragos Oprea in 2007,  and also Prakash Belkale in 2009.)

Now let $X$ be a smooth complex projective surface.  In general the locus $D_{\lambda}$ in $(\ref{divisor})$ might not be a divisor in $\mhu\times\mhc$.  However,  when $D_{\lambda}$ is a divisor,  is the induced morphism $D$ an isomorphism?  This is the question proposed by Strange Duality conjecture for surfaces and so far only few special cases are known.   For instance Danila proves that Strange Duality holds for $u=(0,dH,\chi(u)=0),c=(2,0,c_2)$ on $\mathbb{P}^2$ for small $c_2$ and $d=1,2,3$ (see \cite{nila});
and Marian and Oprea prove that it holds in a large number of cases for generic K3 surfaces (see \cite{mosec}).

Strange Duality also proposes a numerical question,  namely,  whether the following equality holds
\begin{equation}\label{numcon}h^0(\mhu,\lc)=h^0(\mhc,\lu)?
\end{equation}

There is another version of Strange Duality on the numerical level.  Instead of the same dimension of the spaces of global sections,  we ask whether the two line bundles have the same Euler characteristic.  In other words,  whether the following equality holds
\begin{equation}\label{ecnum}\chi(\mhu,\lc)=\chi(\mhc,\lu)?
\end{equation}

Some people think that Strange Duality conjecture might be too "good" to be always true and Equation (\ref{ecnum}) seems more reasonable than Equation (\ref{numcon}).  In fact in all the cases that are known so far,  both $\lc$ and $\lu$ have no higher cohomology.  Hence in those cases $(\ref{numcon})$ and $(\ref{ecnum})$ are both true. 

In this paper we let $X$ be a smooth complex projective surface,  and most of the time $X=\mathbb{P}^2$ or $\mathbb{P}(\mo_{\pone}\oplus\mo_{\pone}(-e))$ with $e=0,1$.  And we let $u$ and $c$ be specified as $u=(0,L,\chi(u)=0)$ for $L$ an effective line bundle on $X$ satisfying some conditions (see Subsection 4.2),  
and $c=c_n^r=r[\mathcal{O}_X]-n[\mathcal{O}_{pt}]$ where $\mo_{pt}$ is the skyscraper sheaf supported at a point in $X.$  In this situation,  the locus $D_{\lambda}$ in (\ref{divisor}) is a divisor in $\mhu\times\mhc$ (see \cite{nila} Theorem 2.1).  Hence we have a morphism $D$ well-defined up to scalars as in (\ref{sdm}). 
According to Strange Duality,  the morphism $D$ is conjectured to be an isomorphism.

We are concerned mainly with the numerical version of Strange Duality.  We would like to check (\ref{numcon}) and (\ref{ecnum}) for our specified $u$ and $c=c^r_n.$

For $r=1,$  $\chi(M_X^H(c^r_n),\lu)$ equals to $h^0(M_X^H(c^r_n),\lu)$ and has been computed in \cite{adv}.  In this paper we get the first equality in this case and moreover the morphism $D$ in (\ref{sdm}) is an isomorphism.

For $r=2$,  on the right hand side of both equations,  it corresponds to studying the Donaldson line bundle on $M_X^H(c^r_n).$  G\"ottsche has computed $\chi(M_X^H(2,c_1,c_2),\lambda_L),$  for $X$ $\mathbb{P}^2$ or any Hirzebruch surface.  And given $n$ large enough,  the higher cohomology groups of $\lambda_L$ will vanish.  

In this paper for $X=\mathbb{P}^2$ or $\mathbb{P}(\mo_{\pone}\oplus\mo_{\pone}(-e))$ with $e=0,1$,  we also compute the generating function 
\begin{equation}\label{generating}Z^r(t)=\sum_{n\geq0}h^0(\mhu,\lcn)t^n.\end{equation}
for any $r\geq 1$ and $g_L\leq2$ with $g_L$ the arithmetic genus of curves in $\ls$.  
We also show that $\lcn$ has no higher cohomologies and hence the results as $r=2$ provide a check for the equality (\ref{ecnum}).  And moreover the equality in (\ref{numcon}) holds for $n$ big enough.

The structure of the paper is as follows.  In section 2,  we briefly review how to define the determinant line bundles.  In section 3,  we give some basic properties of the moduli space $\mhu$ and the line bundle $\lcn$.  In section 4,  we compute the generating function (\ref{generating}).   We divide section 4 into several subsections,  we deal with the case $g_L\leq0$ in the first subsection.  Then we specify $X$ to be $\mathbb{P}^2$ or $\mathbb{P}(\mo_{\pone}\oplus\mo_{\pone}(-e))$ with $e=0,1$ and in the second subsection we prove the result when $r=1$ and also we get an interesting corollary in compactified Jacobian theory of integral planar curves.   Finally the cases as $r\geq2$ and $g_L=1,2$ are studied in the last two subsections.

\section{Preliminaries}
For any $X$ projective scheme of dimension $d$ over an algebraically closed field $\emph{k}$ of characteristic zero,  with $H$ the very ample divisor on it and $u\in \hk(X),$  one can construct a moduli space $\mhu$ corepresenting the functor of semi-stable sheaves with respect to $H$ of class $u$,  as a categorical quotient of some open subscheme $\omu$ in a $\textbf{Quot}-$scheme by a reductive group.  (see \cite{dan} chap. $4$.)  

\begin{equation}\label{quot}
\xymatrix{
\omu\ar[r]^{\phi_u}&\mhu,}
\end{equation}
where $\phi_u$ is a categorical quotient.

Let $c$ be another element in $\hk(X)$ satisfying (\ref{bot}),  then there is a well-defined so-called determinant line bundle $\lc$ on $\mhu$ obtained by descending a line bundle $\tilde{\lc}$ on $\omu.$  $\tilde{\lc}$ is defined as the image of $c$ through $\tilde{\lambda}$ which is the composition of the following homomorphisms:
\begin{equation}\label{dlb}
\xymatrix@C=0.3cm{
  \hk_0(X) \ar[rr]^{q^{*}} && \hk^0(\omu\times X) \ar[rr]^{.[\me]} && \hk^0(\omu\times
X) \ar[rr]^{~~~p_{!}} && \hk^0(\omu)\ar[rr]^{det^{-1}} &&
\textbf{Pic}(\omu),}\end{equation}
where $[\mf].[\mg]=\sum_p (-1)^pTor^p(\mf,\mg)$,
$p_{!}([\mf])=\sum_i(-1)^i[R^ip_{*}\mf];$  and $\me$ is a universal sheaf on $\omu \times X.$  Although the universal sheaf is not unique,  it  won't cause ambiguity because of (\ref{bot}). (also see \cite{dan} chap. $8$.)  

Notice that when $X$ is a simply connected surface,  i.e. $H^1(\mo_X)=0,$  both the moduli spaces and the line bundles only depend on the images of $u$ and $c$ in $\hk(X)_{num}.$  Here $\hk(X)_{num}$ is the Grothendieck group modulo numerical equivalence.
\section{Determinant line bundles $\lcn$ on $\mhu$}
Let $X$ be a smooth complex projective surface,  and $u=(0, L, \chi(u)=0),$  where $L$ is an effective line bundle on $X$ and $\chi(u)$ is the Euler characteristic.   Let $c^r_n=r[\mo_X]-n[\mo_{pt}]$ with $r\geq1$ with $\mo_{pt}$ the skyscraper sheaf supported at a point in $X.$  Then we have $u\perp c^r_n$ for all $r$ and $n.$  One sees that the map $\tilde{\lambda}$ defined in (\ref{dlb}) is a group homomorphism.  So we have $\lcn\simeq \z^{\otimes r}\otimes \lambda_{pt}^{\otimes -n},$  where $\z$ and $\lambda_{pt}$ are line bundles obtained by descending $\tilde{\lambda}(\mo_X)$ and $\tilde{\lambda}(\mo_{pt})$ on $\omu$ respectively.   We have the following lemma for the determinant line bundles on $\mhu.$

\begin{lemma}\label{ngs}For any $c\perp u$ with $c$ of positive rank,  choose a representative torsion-free sheaf $\mg\in c$,  then 

1, $Tor^i(\mf,\mg)=0$ for all $i>0$ and $[\mf]\in\mhu;$

2, there is a natural global section of $\lambda_c$ whose zero set consists of all points $[\mf]$ that $h^0(\mg\otimes\mf)\neq0.$ 

Moreover $\z$ has a natural global section whose zero set consists of all the points $[\mf]$ that $h^0(\mf)\neq0.$  
\end{lemma} 
\begin{proof}See \cite{nila}, Lemma $2.3$ and Lemma $2.4$.
\end{proof}

It is easy to define a set map $\pi$ from the moduli space $\mhu$ to the linear system $\ls,$  satisfying the condition $\textbf{C}$ as follows.

$\textbf{C}:$Generally when the sheaf $E$ is of rank $1$ at its support,  $\pi$ sends it to its support.  When $E$ is of higher rank at some component,  its image is a curve having nonreduced structure at that component and the multiplicity is equal to the rank.  

Since all sheaves in $\mhu$ have the same first Chern class,  $\pi$ is well-defined a priori as a set map.   We have the following proposition.
\begin{prop}\label{morph}$\pi:\mhu\ra\ls$ is a morphism. \end{prop}
\begin{proof}We first show that there is a morphism $\tilde{\pi}:\omu\ra\ls$ satisfying condition $\textbf{C}$ above.  And because $\tilde{\pi}$ is invariant under the group action,  it factors through the quotient $\phi_u$ and gives the morphism $\pi.$

By Lemma \ref{llfr} in Appendix A,  we have a locally free resolution on $\omu\times X$ for the universal sheaf $\me.$
\[\xymatrix{0\ar[r]&\ma\ar[r]^{\alpha}&\mb\ar[r]&\me\ar[r]&0,}\]
where $\ma$ and $\mb$ are locally free sheaves.

Since $\me$ is a family of sheaves of dimension $1$,  the map $\alpha$ is locally given by a square matrix.  So we can define a section of the line bundle $(det \ma)^{-1}\otimes det\mb $ as $det~\alpha.$  The divisor given by this section defines a subscheme in $\omu\times X,$  which induces a morphism $\tilde{\pi}$ from $\omu$ to $\ls$ satisfying condition $\textbf{C},$  this is because $\ls$ represents the functor of curves in class $L$ in $X.$  Hence the proposition.
 \end{proof}
  
We call the image of $[\mf]\in\mhu$ through $\pi$ the \textbf{schematic support} of $\mf$.      
Since we have a morphism $\pi:\mhu\ra\ls$,  it is natural to study the moduli space via $\pi.$  Especially we have the following proposition due to Le Potier (\cite{le}, Proposition $2.8$).

\begin{prop}\label{jlp}If $\mo_{pt}$ is not supported at a base point of $\ls,$  then  $\lambda_{pt}\simeq
\pi^{*}\mo_{\ls}(-1).$
\end{prop} 
Let $c_n^r=r[\mo_X]-n[\mo_{pt}]$ with the skyscraper sheaf $[\mo_{pt}]$ generic so that it is not supported at any base point of $\ls.$  We then have $\lcn\simeq\z^r\otimes\pi^{*}\mo_{\ls}(n).$  From now on we write $\z^r\otimes\pi^{*}\mo_{\ls}(n)$ as $\z^r(n)$ for short.

\begin{prop}\label{ample}$\z^r(s)$ is ample on $\mhu$ for any positive $r$ and $s\gg0.$
\end{prop}
\begin{proof}Proposition 7.6 and Proposition 7.7 in \cite{new} imply that $\lc$ is ample on $\mhu$ for 
\[c=P(n)[\mo_X(mH)]-P(m)[\mo_X(nH)]=(m-n)H.L[\mo_X]-\frac12(H.Lnm(m-n))[\mo_{pt}],\]
where $P(n)=H.Ln$ is the Hilbert polynomial of sheaves in $u$ respect to a very ample line bundle $H$,  and $m>n\gg0.$  We see that $\lc\simeq\z^{(m-n)H.L}(\frac12(H.Lnm(m-n))).$  Moreover $\pi^{*}\mo_{\ls}(1)$ is nef and hence $\z^r(s)$ is ample for $s\gg0.$
\end{proof}
\section{Main results}

Fix $u$ and $c$ as in previous section,  and let $\mo_{pt}$ not be supported at the base points of $\ls,$  we consider the generating function for any fixed $r\geq 1:$ 
\[Z^r(t)=\sum_{n}h^0(\mhu,\lcn)t^n.\]
Because of Proposition \ref{jlp} above we can write $Z^r(t)$ as 
\[Z^r(t)=\sum_{n}h^0(\mhu,\z^r(n))t^n.\]

Since $h^0(\mhu, \lcn)=h^0(\ls,\pi_{*}\lcn),$  and $\pi_{*}\lcn\simeq \pi_{*}(\z^r)\otimes\mo_{\ls}(n),$  we have 
\[Z^r(t)=\sum_{n}h^0(\ls,\pi_{*}(\z^r)\otimes\mo_{\ls}(n))t^n.\]

Since $\mhu$ is projective,  $\pi_{*}(\z^r)$ is coherent on $\ls$ and hence this sum is bounded below,  i.e.  for every fixed $r$,  $\exists N^r_0\in \mathbb{Z},$  such that $h^0(\mhu,\lcn)=0,$  for all $n<N_0^r.$ 

For any two divisor classes $L$ and $L'$,  we write $L'\leq L$ if $L-L'$ is an effective class,  i.e. $h^0(L-L')\neq0;$  and write $L'<L$ if $L'\leq L$ and $L'\neq L.$  Let $g_L$ (resp. $g_{L'}$) be the arithmetic genus of curves in $\ls$ (resp. $|L'|$).  

\subsection{Non-positive genus cases}
In this subsecton we consider the case when $g_L\leq0.$  We have the following description for the moduli space $\mhu$ over $\ls$ and the determinant line bundles via the canonical morphism $\pi:\mhu\ra\ls.$

\begin{prop}\label{isomls}Let $\ls$ be a linear system satisfying the following condition:

For any $0<L'\leq L,$ we have $g_{L'}\leq0,$  i.e. every curve in $\ls$ does not contain any subscheme with positive genus;

Then $\mhu\simeq \ls$,  $\z\simeq \mo_{\ls},$  and every closed point $p$ in $\mhu$ corresponds to a class S-equivalent to $\mo_{\mathbb{P}^1}(-1)^{\oplus N_C} $ with $N_C$ the number of irreducible components counting the multiplicity if it is not reduced,  of the curve $[C]=\pi(p)\in\ls$.
\end{prop}
\begin{proof}In order to prove $\mhu\simeq\ls,$  it is enough to show that in this case $\pi:\mhu\ra\ls$ is bijective.  

Let $C$  be an arbitrary curve in $ \ls,$  let $\{C_i\}$ be the collection of its irreducible components with reduced structure,  i.e. $C_i\simeq\pone$ since $C$ has no subscheme of positive genus.  We then can write $C=\sum_im_iC_i$ as a divisor and let $N_C:=\sum m_i$.  

It is enough to show that for any curve $[C]\in\ls,$  $\pi^{-1}([C])$ is only one point corresponding to the S-equivalence class of $\bigoplus_{i}\mo_{C_i}(-1)^{\oplus m_i}.$  And then by Lemma \ref{ngs},  the claim on $\z$ will follow since $h^0(\mo_{\mathbb{P}^1}(-1))=0.$

We want to prove the following statement:  for any $\mf$ semi-stable sheaf of $\mo_{C}$-modules,  which is of Euler characteristic zero,  $\mf$ is S-equivalent to $\bigoplus_i\mo_{C_i}(-1)^{\oplus n_i}$, 
with $n_{i}[C_i]$ the first Chern class of $\mf.$  

Let $\overline{C}=\sum min(n_i,m_i)C_i.$  Notice that if $\overline{C}\subsetneq C$,  then $\mf$ is a $\mo_{\overline{C}}$-module and then we can reduce to $\overline{C}.$  So with no loss of generality,  we assume that $n_i\geq m_i.$  And the case $n_i=m_i$ proves the proposition.

Now we prove that statement.  We first assume that $[C]$ is connected.  We then do the induction on $N_C$.

If $N_C=1,$  then the statement is trivial since $C=\pone$ and every semi-stable sheaf $\mf$ of Euler characteristic zero and first Chern class $n[C]$ is isomorphic to $\mo_{\pone}(-1)^{\oplus n}$.  

Assume when $N_C\leq n,$  all semi-stable sheaves $\mf$ with $\chi(\mf)=0$ on $C$ are S-equivalent to $\bigoplus_i\mo_{C_i}(-1)^{\oplus n_i}$ with $n_{i}[C_i]$ the first Chern class of $\mf.$  Then let $C$ have $N_C=n+1$ irreducible components counting with multiplicity.   There must exist one integral subscheme $C'\simeq \pone$ in $C$ such that $L'.(L-L')\leq1$ with $L'$ the class of $C'.$  This is because,  if every integral curve in $C$ intersected the closure of its complement at no less than two points,  $C$ would have a "circle" and hence there would be a subscheme of $C$ with positive genus,  which contradicts the given assumption.  
Let $\xi=L'.(L-L')$ and $C''=\overline{C-C'}.$  Then we have these two exact sequences:
\begin{equation}\label{nonre}\xymatrix@C=0.3cm{
 0 \ar[rr]&&\mathcal{O}_{C'}(-\xi)
 \ar[rr]&& \mathcal{O}_{C}
\ar[rr] &&\mathcal{O}_{C''}\ar[rr]&&0;}\end{equation}
\begin{equation}\label{nonred}\xymatrix@C=0.3cm{
 0 \ar[rr]&&I
 \ar[rr] && \mathcal{O}_{C}
\ar[rr] &&\mathcal{O}_{C'}\ar[rr]&&0.}\end{equation}

In the second sequence $I$ is defined as the kernel.  We do not write $\mo_{C''}(-\xi)$ but $I$ instead because $C''$ may not have all the points linearly equivalent. 
Let $\mf\in \pi^{-1}([C]),$  we then
tensor those two exact sequences by $\mathcal{F}$,  we get
\begin{equation}\label{vnonredu}\xymatrix@C=0.3cm{
 \mathcal{F}|_{C'}(-\xi)\ar[rr]^{a} && \mathcal{F}
\ar[rr] &&\mathcal{F}|_{C''}\ar[rr]&&0;}\end{equation}
\begin{equation}\label{vnonreduc}\xymatrix@C=0.3cm{
\mathcal{F}\otimes I\ar[rr] && \mathcal{F}
\ar[rr]^{b} &&\mathcal{F}|_{C'}\ar[rr]&&0;}\end{equation}
The morphism $b$ is not zero because it is just a restriction.   Let $\mf'$ be the torsion-free part of $\mf|_{C'},$  then $\mf'$ is a direct sum of line bundles on $C'\simeq\pone$.  And since every direct summand $\mo_{C'}(\mu)$ of $\mf'$ is a quotient of $\mf$,  $\mu\geq-1$ because of the semi-stabiblity of $\mf.$

The morphism $a$ factors through $\mf|_{C'}(-\xi)\ra\mf'(-\xi)$ since $\mf$ is pure.  If $a\neq0,$  then there exists one summand $\mo_{\pone}(\mu)$ of $\mf'$ such that $\chi(\mo_{\pone}(\mu-\xi))\leq 0$,  hence $\mu\leq-1+\xi$ and hence either $\xi=0$ and $\mu=-1$ or $\xi=1$ and $\mu=0.$  And hence either there is a quotient $\mg$ of $\mf$ and $\mg\simeq\mo_{\pone}(-1)$ or there is a subsheaf $\mg'$ of $\mf$ such that $\mg'\simeq\mo_{\pone}(-1)$.  For the first case,  we have 
\[0\ra\mathcal{K}\ra\mf\ra\mg\ra0.\]
Because every subsheaf of $\mathcal{K}$ can not have positive Euler characteristic,  it is semi-stable of Euler characteristic $0.$  And hence $\mf$ is S-equivalent to $\mathcal{K}\oplus\mo_{\pone}(-1).$  But $\mathcal{K}$ is supported at curve $C''$ with $N_{C''}<N_{C},$  so by induction we are done.  For the second case when $\mf$ has a subsheaf isomorphic to $\mo_{\pone}(-1),$  it is analogous.

If $a$ is zero,  that means that $\mf$ is actually a $\mo_{C''}$-module,  then we can use the induction assumption to get the result.  And thus we have proven the proposition for $C$ connected.

And when $C$ has more than one connected component,   the conclusion follows immediately from the fact that $\mf$ is the direct sum of its restrictions to every connected component.   Hence we have proven the whole proposition.
\end{proof}

Proposition \ref{isomls} applies to these following examples:
\begin{example}\label{exone}$X=\mathbb{P}^2.$  Denote $H$ the hyperplane.  Then we let $L=dH$ with $d=1,2$.  
\end{example}  

\begin{example}\label{extwo}Let $X$ be any Hirzebruch suface,  i.e. $X=\mathbb{P}(\mo_{\mathbb{P}^1}\oplus\mo_{\mathbb{P}^1}(-e))$ for some $e\geq0.$  Let $F$ be the fiber class and $G$ the section class with $G.G=-e$.  Then $L=nF,nG,nF+G,$ for any $n\geq1.$

And on the blow-up $\hat{X},$  denote $\hat{F}$ (resp. $\hat{G}$) the pull back of $F$ (resp. $G$) and $E$ the exceptional divisor.  Then $L=\hat{F},\hat{G}, 2\hat{F}-E,2\hat{G}-E, \hat{F}+\hat{G}-E.$  
\end{example}

In this case,  the generating function can be written down as
\begin{eqnarray}Z^r(t)&=&\sum_{n}h^0(\mhu,\lcn)t^n\nonumber\\
&=&\sum_{n}h^0(|L|,\Theta^r\otimes\mathcal{O}_{\ls}(n))t^n\nonumber
\\
&=&\sum_{n}h^0(\ls, \mathcal{O}_{\ls}(n))t^n\nonumber\\
&=&\large{\frac{1}{(1-t)^{l+1}}}.\nonumber
\end{eqnarray}
where $l=dim~\ls.$

Obviously $\Theta^r(n)$ has no higher cohomologies for $n\geq0.$  And the formula is exactly what we expect and it matches G\"ottsche's result on the side of rank 2 sheaves.  Hence we have  
\begin{coro}Let $X$ be the projective plane or some Hirzebruch surface,  and let $L$ be an effective line bundle on $X$ as in Proposition \ref{isomls}.  Let $u=(0,L,\chi(u)=0)$ and $c_n=(2,0,n)$,  then we have for all $n\geq0$
\[\chi(M(c_n),\lu)=\chi(M(u),\lambda_{c_n})=h^0(M(u),\lambda_{c_n}).\]
\end{coro}

\subsection{Some properties of $\mhu$}
In this subsection we are going to prove some properties of the moduli space $\mhu$.  Those results provide part of the key ingredient for our later argument,  when we deal with the positive genus cases. 

Now we let the surface $X$ and the effective line bundle $L$ be as in Example $\ref{exmone}$ and $\ref{exmtwo}$ as follows.
  
\begin{example}\label{exmone}Let $X=\mathbb{P}^2,$   and $H$ be the hyperplane.  Let $L=dH$ with $d\geq 3$.
\end{example}  
\begin{example}\label{exmtwo}Let $X=\mathbb{P}(\mo_{\mathbb{P}^1}\oplus\mo_{\mathbb{P}^1}(-e))$ for $e=0,1,$  with $F$ the fiber class and $G$ the section class.   $G.G=-e$.  Let $L=2G+nF,$ for any $n>max\{1,2e\}.$ 
\end{example}

Let $K$ be the canonical divisor on $X.$  Denote $l$ to be $dim~\ls$,  $g_L$ the arithmetic genus of curves in $\ls.$  As one can see,  we always have $g_L>0.$  

We have a natural morphism $\pi:\mhu\ra\ls$ sending every sheaf to its schematic support.  It is easy to see that fibers of $\pi$ over integral curves are of dimension $g_L,$  but fibers over non-integral curves might not be of dimension $g_L.$  Let $\ls^{int}$ denote the biggest open subscheme in $\ls$ formed by the points where fibers of $\pi$ have dimension $g_L$.  Of course,  $\ls^{int}$ contains all points corresponding to integral curves. 

We can see that $L$ satisfies three conditions as follows:

$(\ha_1)$ There is a very ample divisor $H$,   such that for any $0<L'\leq L,$ either $L'.(K+H)<0$ or $L'=G$ or $2G$ on $X=\mathbb{P}(\mo_{\mathbb{P}^1}\oplus\mo_{\mathbb{P}^1}(-1))$. 

$(\ha_2)$  For any $0<L'\leq L$,  if $g_{L'}\leq0$ then any curve in $|L'|$ contains no subscheme of positive genus;  and moreover for any collection of effective line bundles  $\{L_i\}$ such that $L=\sum L_i$,  we have $\sum l_i+\sum max\{g_{L_i},0\}+2\leq l+g_L$ with $l_i$ the dimension of $|L_i|.$ 

$(\ha_3)$ There are connected smooth curves in $\ls$,  and non-integral curves are contained in a subset of codimension $2$ in $\ls$.  Hence $\ls-\ls^{int}$ is of codimension at least $2$ in $\ls$. 

From now on for simplicity let $M=\mhu$ and $M^s=\mhu^s.$
\begin{lemma}\label{mainzero}The moduli space $M$ is normal and Cohen-Macaulay,  and its stable locus $M^s$ is smooth of dimension $L.L+1=l+g_L$. 
\end{lemma}
\begin{proof}Let $\mf$ be a semistable sheaf of class $u$.  From condition $(\ha_1)$ we know that $L'.K<0$ for all $0<L'\leq L$,  which together with the semistability of $\mf$ implies the vanishing of Ext$_0^2(\mf,\mf),$  since Ext$^2(\mf,\mf)^{\vee}\simeq$ Hom$(\mf,\mf\otimes K)=0.$  Hence we have the smoothness of both $M^s$ and the Quot-scheme $\omu$.   The whole moduli space $M$ is normal and Cohen-Macaulay because it is a good quotient of a smooth scheme by a reductive group (see \cite{bout}).  To get the dimension is just a direct computation.
\end{proof}

\begin{rem}Because sheaves in $M$ are torsion sheaves with rank zero,  the trace map $tr:$ Ext$^1(\mf,\mf)\ra H^1(\mo_X)$ may not be surjective if $H^1(\mo_X)$ is not zero.  Then it will be more difficult to compute the dimension of $M^s$.
\end{rem}
\begin{rem}\label{chionesm}Lemma \ref{mainzero} also holds for $M_n:=M(u_n)$ with $u_n=(0,L,\chi(u_n)=n)$ and $n$ any integer. 
\end{rem}
\begin{lemma}\label{mainze}The strictly semi-stable locus,  i.e. $M-M^s,$   is of codimension at least $2.$  
\end{lemma}
\begin{proof}
To see $M-M^s$ is of codimension at least $2,$  we can just follow Le Potier's method to prove Proposition $3.4$ in \cite{lee}:  there is a injective map from $M-M^s$ to $\bigcup_{\sum u_i=u} (\prod_{i} M^H_X(u_i)^s$),  where $u_i=(0,L_i,\chi(u_i)=0)$ for some effective $L_i$,  and the union is taken over all the collections $\{u_i\}$ such that $\sum L_i=L.$  There are finitely many of such collections.  So the condition $(\ha_2)$ and Lemma \ref{mainzero} imply that $M-M^s$ is of codimension at least $2$.  
\end{proof}

Recall the quotient in (\ref{quot})
\begin{equation}\label{quotch}
\xymatrix{
\Omega_u\ar[r]^{\phi_u}&M,}
\end{equation}
where $\phi_u$ is a good quotient by some reductive group $G.$  Denote $\omu^{sm}$ to be the open subscheme of $\omu$ consisting of all the sheaves that are locally free on their supports.  Let $M^{sm}$ be the intersection of $M^s$ with the image of $\omu^{sm}$ through $\phi_u.$  Notice that there might be strictly semi-stable sheaves that are locally free on their supports,  hence $M^{sm}\subsetneq \phi_u(\omu^{sm})$ in general.

It is proven by Le Potier (\cite{lee},  Proposition 2.8 and Proposition 2.9) that $\pi\circ\phi_u$ restricted on $\omu^{sm}$ is smooth and $\pi$ restricted on $M^{sm}$ is smooth.  

\begin{lemma}\label{lep}$\omu-\omu^{sm}$ is of codimension no less than $2$ in $\omu.$  Hence $\omu^{sm}$ is dense in $\omu.$ \end{lemma}
\begin{proof}
It is also proven by Le Potier (\cite{lee},  Lemma 3.2) that $\omu-\omu^{sm}$ is of codimension $2$ when the surface is $\mathbb{P}^2.$  However we can just follow his method and finally get that:  For every closed point $[\mathcal{H}\ra\mf]\in\omu$ with $\mf$ not locally free on its support,  if the sheaf $\mathcal{E}xt^1(\mf,\mf)$ is globally generated,  then it is contained in a subset of $\omu-\omu^{sm}$ which is of codimension at least $2$ in $\omu$.

By Castelnuovo-Mumford Theorem (\cite{mum};  or \cite{dan},  Lemma 1.7.2),  if for all $i>0$,  $H^i(\mathcal{E}xt^1(\mf,\mf)(-iH))=0$ with $H$ some very ample line bundle,  then $\mathcal{E}xt^1(\mf,\mf)$ is globally generated.  Since $\mf$ is of dimension $1$,   we know that $H^j(\mathcal{H}om(\mf,\mf(-iH)))=0$ and $H^j(\mathcal{E}xt^1(\mf,\mf)(-iH))=0$ for all $i$ and $j>1.$  And also by Lemma \ref{llfr} in Appendix A,  we know that $\mf$ has a locally free resolution of length $1$,  hence $\mathcal{E}xt^j(\mf,\mf(-iH))=0$ for all $i$ and $j>1.$  Therefore by the spectral sequence we have Ext$^2(\mf,\mf(-H))\simeq H^1(\mathcal{E}xt^1(\mf,\mf(-H))).$

By Serre Duality,  Ext$^2(\mf,\mf(-H))\simeq$Hom$(\mf,\mf(H+K))^{\vee}.$  For $X=\mathbb{P}^2$ or $\mathbb{P}^1\times\mathbb{P}^1$,  Condition $(\ha_1)$ says that for all $0<L'\leq L$,  $L'.(H+K)<0$ which together with the semi-stability of $\mf$ lead to the vanishing of Hom$(\mf,\mf(H+K))$.  And hence we know that $\omu-\omu^{sm}$ is of codimension at least 2 in $\omu$ for $X=\mathbb{P}^2$ or $\mathbb{P}^1\times\mathbb{P}^1$.

Let $X=\mathbb{P}(\mo_{\mathbb{P}^1}\oplus\mo_{\mathbb{P}^1}(-1))=:\hat{\mathbb{P}^2}$,  i.e.  $X$ is obtained by blowing up $\mathbb{P}^2$ at a point.  Let $L=2G+nF$ with $n\geq3$ and $H=G+2F$ very ample.  We only need to show that all closed points $[\mathcal{H}\ra\mf]\in\omu$ with Hom$(\mf,\mf(H+K))\neq0$ are contained in some subset of codimension at least $2.$  
 
If Hom$(\mf,\mf(H+K))\neq0$,  then $\mf$ must contain a subsheaf $\mf'$ semistable of Euler characteristic zero such that Hom$(\mf,\mf'(H+K))\neq0$.  Then we must have $c_1(\mf')=i[C_G]$ with $C_G$ the only curve in class $G$ and $i=1$ or $2$.  By Proposition \ref{isomls},  we know that every semistable sheaf of Euler characteristic zero on $2G$ is $S$-equivalent to $\mo_{C_G}(-1)\oplus\mo_{C_G}(-1)$.  Hence with no loss of generality,  we assume $\mf'$ is supported at the curve $C_G\simeq \pone$ and $\mf'\simeq \mf'(H+K)\simeq\mo_{C_G}(-1)$.  Then we have the following exact sequence.
\begin{equation}\label{febu}0\ra\mo_{C_G}(-1)\ra\mf\ra\mg\ra0.
\end{equation}
Since Hom$(\mf,\mo_{C_G}(-1))\neq0$,  either $\mo_{C_G}(-1)$ is a direct summand of $\mf$ then sequence (\ref{febu}) splits or there is a nonzero morphism $\mg\ra\mo_{C_G}(-1)$.  We define two subsets of $\omu-\omu^s$ as follows.
\[\Sigma_1:=\{[\mathcal{H}\ra\mf]\in\omu|\mo_{C_G}(-1)~is~a~direct~summand~of~\mf.\}\]
\[\Sigma_2:=\{[\mathcal{H}\ra\mf]\in\omu|with~\mf~in~sequence~(\ref{febu})~and~Hom(\mg,\mo_{C_G}(-1))\neq0\}.\]

It will suffice for proving the rest of the lemma to show that both $\Sigma_1$ and $\Sigma_2$ are of codimension at least $2$ in $\omu.$  It is easy to compute that 
$dim~\Sigma_1=dim~\omu-G.(G+nF)$,   hence $\Sigma_1$ is of codimension $\geq2$ in $\omu$ as $n\geq3.$ 

Now we estimate the dimension of $\Sigma_2.$  In sequence $(\ref{febu})$ the sheaf $\mg$ is semistable and $c_1(\mg)=G+nF.$  Moreover since $\mo_{C_G}(-1)$ is stable,  every nonzero morphism $\mg\ra\mo_{C_G}(-1)$ must be surjective.  Thus $\mg$ must lie in the following sequence
\begin{equation}\label{febutwo}0\ra\mk\ra\mg\ra\mo_{C_G}(-1)\ra0,
\end{equation} 
where $\mk$ is semi-stable of Euler characteristic zero and $c_1(\mk)=nF.$  
By Lemma \ref{finitegz} in Appendix B,  the space of isomorphic classes of such $\mk$ has dimension equal to $n=dim~|nF|.$

For the fixed $\mk$,  all the isomorphic classes of $\mg$ in (\ref{febutwo}) form a space of dimension no larger than $dim~$(Ext$^{1}(\mo_{C_G}(-1),\mk))/\mathbb{G}_m$.  Since for $i=0,2,$ Ext$^i(\mo_{C_G}(-1),\mk)=0$,  $dim~$(Ext$^{1}(\mo_{C_G}(-1),\mk)/\mathbb{G}_m)$=$G.nF-1=n-1$. 

Fix the sheaf $\mg$,  then the different choices of $\mf$ in (\ref{febu}) form a space of dimension at most $dim~($Ext$^1(\mg,\mo_{C_G}(-1)))/\mathbb{G}_m$.  Notice that Ext$^2(\mg,\mo_{C_G}(-1))=0$,  and $dim~$Hom$(\mg,\mo_{C_G}(-1))=1$.  Hence $dim~($Ext$^1(\mg,\mo_{C_G}(-1)))/\mathbb{G}_m=G.(G+nF)+1-1=n-1.$  

Finally we know that $dim~\Sigma_2\leq dim~$Hom$(\mathcal{H},\mathcal{H})-1+n+(n-1)+(n-1)$.  On the other hand we have $dim~\omu=dim~$Hom$(\mathcal{H},\mathcal{H})-1+g_L+l$.  By a direct computation,  we get $g_L+l=4n-3$.  Hence $dim~\Sigma_2\leq dim~\omu-(n-1)$,  and hence the lemma as $n\geq3.$ 
\end{proof}
\begin{lemma}\label{oirred}$\omu$ is irreducible.
\end{lemma}
\begin{proof}Since $\omu^{sm}$ is dense in $\omu$ by Lemma \ref{lep},  it is enough to show that $\omu^{sm}$ is irreducible.  Since $\omu^{sm}$ is smooth,  it is enough to show that it is connected.  We assume that $\omu^{sm}=U_1\cup U_2$ with $U_1$ and $U_2$ are open and $U_1\cap U_2=\emptyset.$  Let $V\subset \ls$ be the open set parametrizing smooth curves.  And obviously $U:=(\pi\circ\phi_u)^{-1}(V)$ is connected and open in $\omu^{sm}.$  Hence $U$ is contained in $U_1$ or $U_2.$  Let $U$ be contained in $U_1,$  then $U_2$ is contained in the preimage of $\ls-V$ in $\omu^{sm}.$  On the other hand,  $\pi\circ\phi_u$ restricted on $\omu^{sm}$ is smooth hence the preimage of $\ls-V$ in $\omu^{sm}$ has dimension less than the dimension of $\omu$.  But $\omu$ is smooth and equidimensional and every nonempty open subscheme of it has the same dimension as it.  Hence we know that $U_2$ has to be empty and $\omu^{sm}$ is connected and hence the lemma.    
\end{proof}
\begin{coro}\label{irred}The moduli space $M$ is irreducible.
\end{coro}
\begin{rem}\label{chione}For any integer $n$,  the stable locus $M^s_n$ of $M_n$ (as defined in Remark \ref{chionesm}) is irreducible. 
\end{rem}
We now study the dualizing sheaf on $M$.

\begin{prop}\label{dualizing}Let $\omega$ be the canonical line bundle on $M^s$,  then $\omega\simeq(\pi^s)^{*}\mo_{\ls}(1)^{\otimes L.K},$  with $\pi^s$ obtained by composing the open embedding from $M^s$ to $M$ with $\pi.$  Moreover,  the dualizing sheaf on $M$ is locally free and isomorphic to $\pi^{*}\mo_{\ls}(1)^{\otimes L.K}.$
\end{prop}
\begin{proof}Since $M$ is normal,  the dualizing sheaf on $M$ is the push forward of the canonical bundle on its smooth locus.  On the other hand $M$ is Cohen-Macaulay and $M-M^s$ is of codimension $\geq2.$  Hence if $\omega\simeq(\pi^s)^{*}\mo_{\ls}(1)^{\otimes L.K},$  then the push forward of $\omega$ to $M$ is the dualizing sheaf on $M$ and isomorphic to $\pi^{*}\mo_{\ls}(1)^{\otimes L.K}.$

It will suffice to prove that $det~(\mathcal{T}_{M^s})=(\pi^s)^{*}\mo_{\ls}(-1)^{\otimes L,K}$,  where $\mathcal{T}_{M^s}$ is the tangent bundle on $M^s.$
   
Restrict the quotient in (\ref{quotch}) on $M^s$ and we get
\begin{equation}\label{quotchi}
\xymatrix{
\Omega_u^s\ar[r]^{\phi^s_u}&M^s.}
\end{equation}
Since $\phi_{u}^s$ is a principal $PG$-bundle ,  we have $(\phi^s_u)^{*}:Pic~(M^s)\ra Pic^{PG}(\Omega_{u}^s)$ is an isomorphism.  And also because there is no surjective homomorphism from $PG$ to $\mathbb{G}_m,$  the natural morphism $Pic^{PG}(\Omega_{u}^s)\ra Pic (\Omega_{u}^s)$ is injective (\cite{git} Chap 1,  Section 3,  Proposition 1.4).  Hence it is enough to prove that 
\begin{equation}\label{tais}det~(\phi_u^s)^{*}\mathcal{T}_{M^s}\simeq(\phi_{u}^s)^{*}(\pi^s)^{*}\mo_{\ls}(-1)^{\otimes L.K}.\end{equation}

Because of Lemma \ref{lep},  we know that it is enough to show the isomorphism in (\ref{tais}) restricted on $\omu^{sm}\cap\omu^s$,  i.e.  we prove that 
\[det~(\phi_{u}^{s})^{*}\mathcal{T}_{M^{s}}|_{\omu^{sm}\cap\omu^s}\simeq(\phi_{u}^{s})^{*}(\pi^{s})^{*}\mo_{\ls}(-1)^{\otimes L.K}|_{\omu^{sm}\cap\omu^s}.\]

We have a universal sheaf on $X\times\Omega_{u}^{s}.$  We denote it as $\mathcal{E}$.  Then
\begin{equation}\label{ppq}
\xymatrix{
  \mathcal{E}  \ar[r]
                & X\times \Omega_{u}^{s} \ar[ld]^{q} \ar[d]^{p}  \\
                X & \Omega_{u}^{s}  
               }
\end{equation}

By Theorem 10.2.1 in \cite{dan},  we have
\[(\phi_{u}^{s})^{*}\mathcal{T}_{M^{s}}=\mathcal{E}xt_p^1(\mathcal{E},\mathcal{E}).\]  

For every closed point $m\in \Omega_{u}^s,$  we have Ext$^i(\mathcal{E}_m, \mathcal{E}_m)=0,$ for all $i\geq 2$ and Ext$^0(\mathcal{E}_m, \mathcal{E}_m)=\mathbb{C}.$  Hence $\mathcal{E}xt_p^0(\mathcal{E},\mathcal{E})=p_{*}\mathcal{H}om(\mathcal{E},\mathcal{E})$ is a line bundle on $ \omu^{s},$  and moreover isomorphic to $\mathcal{O}_{\Omega_{u}^s}$ because it has a global section non-vanishing everywhere.  $\mathcal{E}xt_p^i(\mathcal{E},\mathcal{E})=0,$  for all $i\geq2,$  because fiberwise they are Ext$^i(\mathcal{E}_m,\mathcal{E}_m).$  Therefore
\begin{equation}\label{bullet} [det~R^{\bullet}(p\circ\mathcal{H}om(\mathcal{E},\mathcal{E}))]=[det~\mathcal{E}xt_p^1(\mathcal{E},\mathcal{E})^{-1}].\end{equation}
By Proposition \ref{qutang} in Appendix B,  we know that 
\[det~R^{\bullet}(p\circ\mathcal{H}om(\mathcal{E},\mathcal{E}))|_{\omu^{sm}\cap\omu^s}\simeq(\phi_{u}^{s})^{*}(\pi^{s})^{*}\mo_{\ls}(-1)^{\otimes L.K}|_{\omu^{sm}\cap\omu^s}.\]
Hence
\[det~(\phi_{u}^{s})^{*}\mathcal{T}_{M^{s}}|_{\omu^{sm}\cap\omu^s}\simeq(\phi_{u}^{s})^{*}(\pi^{s})^{*}\mo_{\ls}(-1)^{\otimes L.K}|_{\omu^{sm}\cap\omu^s}.\]
And this finishes the proof of the lemma. 
\end{proof}

\begin{rem}\label{chion}Proposition \ref{dualizing} holds for $M_n$,  as long as $M_n-M_n^s$ is of codimension $\geq2$ in $M_n$. 
\end{rem}  
We then have a result in the theory of Compactified Jacobians of integral curves as a corollary of Proposition \ref{dualizing}.

\begin{coro}\label{trsame}Let $X$ and $L$ be as before,  i.e. as in Example \ref{exmone} and \ref{exmtwo}.  Let $C$ be any integral curve in $\ls$.  Then on the compactified Jacobian $\bar{J_d}$ which parametrizes pure sheaves of rank $1$ of degree $d$ on $C$,  the dualizing sheaf $\omega^0$ is trivial. 
\end{coro}  
\begin{proof}For any integral curve $C$ in $\ls,$  the fiber of $\pi$ over $[C]$ is a complete intersection by $l$ divisors in $|\pi^{*}\mo_{\ls}(1)|$ in the smooth locus $M^s$ and it is isomorphic (not canonically) to the compactified Jacobian $\bar{J_d}$ of $C.$  Hence by Proposition \ref{dualizing} we have the lemma.
\end{proof}

\begin{rem}It has been proven by Altman, Kleiman and Iarrobino that on the compactified Jacobian of an integral locally planar curve,  the dualizing sheaf is invertible.  But in general it is not known whether it is trivial. 
\end{rem}

\subsection{Positive genus cases and $r=1$.}
Let $M$ and $\z$ be the same as before.   
Then we have
\begin{thm}\label{main}$1,$ $R^1\pi_{*}\z^r=0,$  for all $r>0.$

$2,$ For all $r>0,$  $\pi_{*}\z^r$ is torsion-free on $\ls,$  and locally free of rank $r^{g_L}$ on $\ls^{int}$.  In particular when $r=1,$  $\pi_{*}\z\simeq \mo_{\ls},$  hence $\pi_{*}\z$ is locally free on the whole linear system $\ls$.
\end{thm}

Statement $2$ of Theorem \ref{main} implies that for all $g_L>0$
\[Z^1(t)=\sum_{n}h^0(M,\lambda_{c^1_n})t^n=\frac{1}{(1-t)^{l+1}}.\]
This coincides with the formula on the side of rank $1$ sheaves.  Thus we have a corollary.
\begin{coro}Let $X$ and $L$ be as in Example \ref{exmone} and \ref{exmtwo}.  Let $u=(0,L,\chi(u)=0)$ and $c_n=(1,0,n)$,  then we have for all $n\geq0$
\[h^0(M(c_n),\lu)= h^0(M(u),\lambda_{c_n}).\]
And moreover the morphism $D$ in (\ref{sdm}) is an isomorphism.
\end{coro}
\begin{proof}To show $D$ is an isomorphism,  it is enough to show that it is injective.  We choose a collection of generators $\{s_i\}$ of $H^0(M,\lambda_{c_n}),$  and we then show that we can find a collection of ideal sheaves $\{I_i\}$ with $[I_i]\in M(c_n)$ for all $i,$  such that for each $i$ the divisor $D_i=\{[\mf]\in M|h^0(\mf\otimes I_i)\neq0\}$ is the zero set of $s_i.$  This will suffice to prove the injectivity of $D.$

Denote $x$ to be a single point in $X$,  let $I_x$ be the ideal sheaf of $x.$  For any $[\mf]\in M$ we have
\[0\ra Tor^1(\mf,\mo_{x})\ra \mf\otimes I_x\ra\mf\ra\mf\otimes \mo_{x}\ra0.\]

Denote $C_{\mf}$ to be the supporting curve of $\mf.$  Notice that:  if $x\not\in C_{\mf}$,  then $Tor^1(\mf,\mo_{x})=0$ and $\mf\otimes I_x\simeq\mf;$  if $x\in C_{\mf},$  then $Tor^1(\mf,\mo_{x})$ is of zero dimension and hence has nonzero global sections.  Hence we know that
\begin{equation}\label{rankis}h^0(\mf\otimes I_x)\neq0\Leftrightarrow h^0(\mf)\neq0~or~x\in C_{\mf}. \end{equation}

Given a point $x\in X$ which is not a base point of $\ls,$  we can define a hyperplane in $\ls$ by asking curves to pass through $x.$  We choose $l+1$ points $x_0,\ldots,x_l\in X,$  such that the corresponding hyperplanes $P_i$ intersect transversely.  Then sections induced by $P_i$'s generate $H^0(\ls,\mo_{\ls}(1))$ and we say that the collection $\Gamma=\{x_0,\ldots,x_l\}$ is regular. 

For $\z(n)$ with $n\geq1,$  we choose $n$ regular collections $\Gamma_i=\{x^i_0,\ldots,x_l^i,\}$ for $1\leq i\leq n,$  such that $\Gamma_i\cap\Gamma_j=\emptyset$ if $i\neq j$.   We take a collection of ideal sheaves $\{I_{\underline{x}}\}$ with $\underline{x}=\{x^1_{i_1},\ldots,x^n_{i_n}\}  .$  We can see that $\#\{I_{\underline{x}}\}=(l+1)^n.$  And denote $s_{\underline{x}}$ to be the section induced by the divisor $D_{\underline{x}}=\{[\mf]\in M|h^0(\mf\otimes I_{\underline{x}})\neq0\}.$ 

We have the following morphism $m$ defined by multiplication
\[m:H^0(M,\z)\otimes H^0(M,\pi^{*}\mo_{\ls}(1))^{\otimes n}\ra H^0(M,\z(n)).\]

Because of (\ref{rankis}),  we know that $s_{\underline{x}}$'s generate the image of $m.$  Hence they generate $H^0(M,\z(n))$ because of Lemma \ref{sur}. 

And hence we have proven the corollary.
\end{proof}

\begin{lemma}\label{sur}$\pi_{*}\mo_M\simeq \pi_{*}\z\simeq\mo_{\ls}$ and the following morphism $m_1$ defined by multiplication is surjective for any $n$
\begin{equation}\label{mul}m_1:H^{0}(M,\pi^{*}\mo_{\ls}(1))\otimes H^0(M,\z(n))\ra H^0(M,\z(n+1)).\end{equation}
\end{lemma}
\begin{proof}Since $\mo_M$ is a subsheaf of $\z,$  we have $\pi_{*}\mo_M$ is a subsheaf of $\pi_{*}\z$ and hence $\pi_{*}\mo_M$ is a subsheaf of $\mo_{\ls}.$  But on the other hand,  $h^0(\ls,\pi_{*}\mo_M)=h^0(M,\mo_M)=1.$  Thus $\pi_{*}\mo_M\simeq\mo_{\ls}.$

On $\ls$ we have that the following morphism $m_1'$ defined by multiplication is surjective for any $n$
\begin{equation}\label{mult}m_1':H^{0}(\ls,\mo_{\ls}(1))\otimes H^0(\ls,\mo_{\ls}(n))\ra H^0(\ls,\mo_{\ls}(n+1)).\end{equation}

And since $\pi_{*}\mo_M\simeq \mo_{\ls},$  the morphism 
\[\pi_{*}:H^0(M,\pi^{*}\mo_{\ls}(1))\ra H^0(\ls,\mo_{\ls}(1))\]
is a canonical isomorphism.   And also the morphism
\[\pi^{\z}_{*}:H^0(M,\z(n))\ra H^0(\ls,\mo_{\ls}(n))\]
is a isomorphism for any $n$.  And hence the surjectivity of $m_1$ in (\ref{mul}) can be deduced from the surjectivity of $m_1'$ in $(\ref{mult})$.  So we have proven the lemma.
\end{proof}

\begin{proof}[Proof of Statement 1 in Theorem \ref{main}]From the spectral sequence we have
\begin{equation}\label{spectral}
H^1(M,\z^r(s))\ra H^0(\ls,R^1\pi_{*}\z^r\otimes\mo(s))\ra H^2(\ls,\pi_{*}\z^r\otimes\mo(s)),\end{equation}

And for any $r>0,$  according to Proposition \ref{dualizing} and the ampleness of $\z^r(s)$ for $s\gg0,$  we know that $H^1(M,\z^r(s))=0$ for $s\gg0.$  And also $H^2(\ls,\pi_{*}\z^r\otimes\mo(s))=0$ for $s\gg0$ because $\pi_{*}\z^r$ is coherent on $\ls.$  Hence we have $H^0(\ls,R^1\pi_{*}\z^r\otimes\mo(s))=0$ for $s\gg0,$  then $R^1\pi_{*}\z^r$ has to be zero.
\end{proof}

To prove Statement $2$ of Theorem \ref{main},  we first show two lemmas.  
Recall that $\ls^{int}$ consists of points over which the fibers are of dimension $g_L$.  We define $M^{int}$ by the following Cartesian diagram
\begin{equation}\label{node}
\xymatrix{
M^{int}\ar[r]^j \ar[d]_{\pi^{int}}&M \ar[d]^{\pi}\\
\ls^{int}\ar[r]^i&\ls}
\end{equation}
Then we have
\begin{lemma}\label{sublf}
$\pi^{int}$ is flat and $\pi^{int}_{*}(j^{*}\z^r)$ is locally free of rank $r^{g_L}$ on $\ls^{int},$  and moreover $R^i\pi^{int}_{*}(j^{*}\z)=0,$  for all $i\geq1.$ 
\end{lemma}
\begin{proof}We have already seen that both $\ls$ and $M$ are irreducilbe.  The flatness of $\pi^{int}$ is because $\ls^{int}$ is regular,  $M^{int}$ is Cohen-Macaulay and every fiber of $\pi^{int}$ is of dimension $g_L$. (See \cite{ha},  III, Ex 10.9.)

Every fiber over a point in $\ls^{int}$ is a complete intersection of $l$ divisors in $|\pi^{*}\mo_{\ls}(1)|.$  Since $H^i(\z^r(s))=0, \forall i>0,s\gg0$,  $\z^r(s)$ restricted to every complete intersection of divisors in $|\pi^{*}\mo_{\ls}(1)|$ has no higher cohomology as $s\gg0$.   But fiberwise $\z^r(s)$ is isomorphic to $\z^r.$  Hence $\z^r$ restricted to every fiber has no higher cohomology,  which together with the flatness of $\pi^{int}$ implies the local freeness of $\pi^{int}\z^r$ (see Theorem 12.11 in Chap. III in \cite{ha}).  Lemma \ref{ngs} implies that $\z|_{\pi^{-1}([C])}$ is the usual $\theta$-bundle on $\pi^{-1}([C])\simeq J_C^{g_L-1}$ and $h^0(\z^r|_{\pi^{-1}([C])})=r^{g_L}$ for $C$ a smooth curve.  Thus $\pi^{int}\z^r$ is of rank $r^{g_L}$.
\end{proof}

\begin{lemma}\label{codm}$M-M^{int}$ is of codimension at least $2$.
\end{lemma}
\begin{proof}Because of Lemma \ref{mainzero},  it is enough to show that $M^s-(M^{int}\cap M^s)$ is of codimension at least $2$ in $M^s.$  

Notice that $M^s-(M^{int}\cap M^s)$ is contained in 
\[(M^{sm}\cap\pi^{-1}(\ls-\ls^{int}))\bigcup(M^s-M^{sm}).\]
So it is enough to prove both the two sets above are of codimension $\geq2$ in $M^s.$

$\pi$ restricted to $M^{sm}$ is smooth.  So condition $(\ha_3)$ implies that $M^{sm}\cap\pi^{-1}(\ls-\ls^{int})$ is of codimension $2.$  And we can easily deduce from Lemma \ref{lep} that $M^s-M^{sm}$ is of codimension $\geq2$.
\end{proof}

\begin{proof}[Proof of Statement $2$ in Theorem \ref{main}]We go back to the Cartesian diagram (\ref{node}).  As we proved in Lemma \ref{sublf},  $\pi^{int}_{*}j^{*}\z^r$ is locally free on $\ls^{int}$ for $r>0.$
Since both $j$ and $i$ in (\ref{node}) are open immersions,  we get $\pi^{int}_{*}j^{*}\z^r\simeq i^{*}\pi_{*}\z^r$ which means that $\pi_{*}\z^r$ restricted to $\ls^{int}$ is locally free for all $r>0.$

Moreover,  $M$ is Cohen-Macaulay,  $\z^r$ is a line bundle on $M$,  and $M-M^{int}$ is of codimension $\geq2$.  Hence according to the theory of cohomology with supports (see \cite{gro} Exp. III, p.8,  Lemma 3.1),  we have $\z^r\simeq j_{*}j^{*}\z^r$ and hence $\pi_{*}\z^r\simeq\pi_{*}j_{*}j^{*}\z^r.$  On the other hand,  because the diagram (\ref{node}) commutes, we have $\pi_{*}\z^r\simeq i_{*}\pi^{int}_{*}j^{*}\z^r.$  We already know that $\pi^{int}_{*}j^{*}\z^r$ is locally free hence torsion-free on $\ls^{int}.$  And $i$ is an open immersion.  So $i_{*}\pi^{int}_{*}j^{*}\z^r$ must be torsion-free on $\ls$.

For $r=1,$  $\pi^{int}_{*}j^{*}\z$ is a line bundle on $\ls^{int}$ with a global section non-vanishing everywhere,  hence $\pi^{int}_{*}j^{*}\z\simeq\mo_{\ls^{int}}= i^{*}\mo_{\ls}.$  And because of $(\ha_3)$ we have that $i_{*}i^{*}\mo_{\ls}\simeq \mo_{\ls}$ and hence $\pi_{*}\z\simeq\pi_{*}j_{*}j^{*}\z\simeq \mo_{\ls}.$  

This finishes the proof of  Statement $2.$
\end{proof}
At the end of this subsection we prove a lemma which gives an estimate of the dimensions of all fibers of $\pi.$  The lemma will be used later.  We recall that the locus 
\[D_{\z}:=\{[\mf]\in M|h^0(\mf)\neq0\}\]
is a divisor of $\z$ on $M$ by Lemma \ref{ngs}.  We have 
\begin{lemma}\label{zdd}For any point $p\in\ls,$  let $D_{p}=D_{\z}\cap\pi^{-1}(p).$  Then we have
\[g_L\leq dim~\pi^{-1}(p)\leq dim~D_{p}+1.\]
\end{lemma}
\begin{proof}Since every fiber is a closed subscheme defined by $l$ equations and $M$ is irreducible of dimension $l+g_L,$  we have $dim~\pi^{-1}(p)\geq g_L.$

To prove $\pi^{-1}(p)\leq dim~D_{p}+1,$  it is enough to show that every irreducible component of $\pi^{-1}(p)$ has dimension no larger than $dim~D_{p}+1.$  Let $F_p$ be an irreducible component of $\pi^{-1}(p).$  If $F_p\cap D_{\z}\neq\emptyset,$  then $F_p\cap D_{\z}$ is a divisor in $F_p$ and hence $dim~F_p=dim~(F_p\cap D_{\z})+1\leq dim~D_{p}+1.$  If $F_p\cap D_{\z}=\emptyset,$  then $\z$ restricted on $F_p$ is isomorphic to the structure sheaf,  so is $\z(n)$ for any $n$ because $F_p$ is contained in a fiber.  But on the other hand $\z(n)$ is ample for $n$ big enough,  hence $F_p$ has to be of dimension zero and hence $dim~F_p<1\leq dim~D_{p}+1.$  So we have proven the lemma.     
\end{proof}

\subsection{Genus one case and $r\geq1.$}
In this subsection we let $g_L=1$ and prove this following theorem: 
\begin{thm}\label{gomrt}
1, $\ls^{int}=\ls$,  hence $\pi$ is flat and $R^i\pi_{*}\z^r=0,  \forall i,r>0;$

2, for $r>0,$
$\pi_{*}\Theta^r\simeq \mathcal{O}_{\ls}\oplus(\mathcal{O}_{\ls}(-i))^{\oplus_{i=2}^r}.$
\end{thm}
Statement $2$ of Theorem \ref{gomrt} implies that for $g_L=1$ we can write down the generating function 
\begin{eqnarray}Z^r(t)&=&\sum_{n}h^0(M,\lcn)t^n
=\sum_{n}h^0(M,\Theta^r\otimes\pi^{*}\mathcal{O}_{\ls}(n))t^n\nonumber
\\
&=&\sum_{n}h^0(\ls, \pi_{*}(\Theta^r)\otimes\mathcal{O}_{\ls}(n))t^n\nonumber\\
&=&\large{\frac{1+t^{2}+t^{3}+\ldots+t^{r}}{(1-t)^{l+1}}}.\nonumber
\end{eqnarray}
When $r=2,$  it matches G\"ottsche's computation.   Hence we have

\begin{coro}Let $X$ and $L$ be as before and moreover $g_L=1$.  Let $u=(0,L,\chi(u)=0)$ and $c_n=(2,0,n)$,  then we have for all $n\geq0$
\begin{equation}\label{eqone}\chi(M(c_n),\lu)=\chi(M(u),\lambda_{c_n})=h^0(M(u),\lambda_{c_n}).\end{equation} 
\end{coro}
\begin{proof}For any $r>0$ and $n\geq0,$  $\z^r(n)$ has no higher cohomology.  This is because $\pi_{*}(\z^r(n))$ has no higher cohomology and $R^i\pi_{*}\z^r=0.$
\end{proof}


\begin{proof}[Proof of Statement 1 in Theorem \ref{gomrt}]Notice that for any $[C]\in\ls,$  the structure sheaf $\mo_{C}$ on $C$ is stable and of Euler characteristic zero.  Hence for any $[\mf]\in D_{\z}$ supported at curve $C$,  $\mf\simeq \mo_{C}$.  And hence we know that $D_{\z}$ restricted to every fiber of $\pi$ is a point and by Lemma \ref{zdd} we know that every fiber of $\pi$ is of dimension $g_L.$  And hence the statement.
\end{proof}

To prove Statement $2$ in Theorem \ref{gomrt},  we need some lemmas.
\begin{lemma}\label{foure}There is a embedding $\imath:\ls\ra M$ induced by the structure sheaf of the universal curve in $X\times \ls.$  Moreover $\imath$ provides a section of $\pi$ with its image the $\z$-divisor $D_{\z},$  where $D_{\z}$ consists of all the $[\mf]$ such that $h^0(\mf)\neq0.$  And hence $D_{\z}\simeq \ls.$
\end{lemma}
\begin{proof}In $X\times \ls$,  there is a universal curve $\mathcal{C}$ such that every fiber $\mathcal{C}_s$ is just the curve represented by point $s$ in $\ls.$ 
\begin{equation}\label{emb}\xymatrix{
  \mathcal{C}  \ar[r]
                & X\times \ls \ar[ld]^{q} \ar[d]^{p}  \\
                X &\ls             }\end{equation}
                
The structure sheaf of $\mathcal{C}$ induces an injective morphism embedding $\ls$ as a subscheme of  $M.$
\[\imath:\ls\rightarrow M.\]

One can see that the image of $\imath$ is $D_{\z}$ and $\imath$ also provides a section of the projection $\pi.$
\end{proof}
\begin{lemma}\label{fourp}$\z^r|_{D_{\z}}\simeq\mathcal{O}_{\ls}(-r)$
\end{lemma}
\begin{proof}According to Lemma $\ref{foure}$ and the universal property of $\z$,  we know that $\z^r|_{D_{\z}}=\imath^{*}\z^r\simeq (det(p_{!}[\mathcal{O}_{\mathcal{C}}]))^{-r},$  with $\mathcal{C}$ and $p$ the same as in (\ref{emb}).

We have an exact sequence on $X\times\ls$.
\[0\rightarrow q^{*}\mathcal{O}_{X}(-L)\otimes p^{*}\mathcal{O}_{\ls}(-1)\rightarrow\mathcal{O}_{X\times\ls}\rightarrow\mathcal{O}_{\mathcal{C}}\rightarrow0.\]

Hence $ (det(p_{!}[\mathcal{O}_{\mathcal{C}}]))^{-1}\simeq  (det(p_{!}[\mathcal{O}_{X\times\ls}]))^{-1}\otimes det(p_{!}[q^{*}\mathcal{O}_{X}(-L)\otimes p^{*}\mathcal{O}_{\ls}(-1)]).$  

And also $det(p_{!}[\mathcal{O}_{X\times\ls}])\simeq \mathcal{O}_{\ls};~$  
$det(p_{!}[q^{*}\mathcal{O}_{X}(-L)\otimes p^{*}\mathcal{O}_{\ls}(-1)])\simeq \mathcal{O}_{\ls}(-1)^{\otimes \chi(\mathcal{O}_{X}(-L))}.$ 

Since curves in $|L|$ are of genus 1,  which means that the structure sheaves of the curves have Euler characteristic $0$,  and hence $\chi(\mathcal{O}_{X}(-L))=\chi(\mathcal{O}_{X})=1.$  So we have $\z^r|_{D_{\z}}\simeq\mathcal{O}_{\ls}(\Theta^r)\simeq \mathcal{O}_{\ls}(-r).$  And hence the lemma.
\end{proof}

\begin{proof}[Proof of Statement $2$ in Theorem \ref{gomrt}]On $M$ we have the exact sequence
\begin{equation}\label{mexa}
0\ra\z^r\ra\z^{r+1}\ra \z^{r+1}|_{D_{\z}}\ra0.
\end{equation}
Push it forward via $\pi$ to $\ls$.  Because of Lemma \ref{foure} and Lemma \ref{fourp},   we have $\pi_{*}\z^{r+1}|_{D_{\z}}\simeq \mo_{\ls}(-r-1)$.  Hence we get
\begin{equation}\label{lsexa}
0\ra\pi_{*}\z^r\ra\pi_{*}\z^{r+1}\ra \mo_{\ls}(-r-1)\ra0.
\end{equation}
We have zero on the right because of Statement $1$ in Theorem \ref{main}.  And by the Statement $2$ in the theorem,  we know that $\pi_{*}\z\simeq \mo_{\ls}.$  Then we have $\z^2\simeq\mo_{\ls}\oplus\mo_{\ls}(-2)$ and by recursion we get the formula for $\pi_{*}\z^r$.  This finishes the proof. 
\end{proof}

\subsection{Genus two cases and $r\geq1.$}
There is no curve of genus two in $\mathbb{P}^2$,  we only have two examples as follow.
\begin{example}\label{exgt}$X=\mathbb{P}(\mo_{\pone}\oplus\mo_{\pone}(-e)),$  with $e=0,1;$  and $L=2G+(e+3)F.$ 
\end{example}
Let $X$ and $L$ be as in Example \ref{exgt} and we have the following theorem:
\begin{thm}\label{gtmrt}1, $\ls^{int}=\ls$,  hence $\pi$ is flat and $R^i\pi_{*}\z^r=0,\forall i,r>0;$

2, for $r>0,$  we have: 

$\pi_{*}\z^r\simeq\mo_{\ls}\oplus\mo_{\ls}(-2)^{\oplus 3}\bigoplus_{i=3}^r (\mo_{\ls}(-i)^{\oplus i+1}\oplus\mo_{\ls}(-i-1)^{\oplus i-2}).$ 
\end{thm}
Statement $2$ of Theorem \ref{gtmrt} implies that for $X$ and $L$ in Example \ref{exgt} we can write down the generating function 
\begin{eqnarray}Z^r(t)&=&\sum_{n}h^0(M,\lcn)t^n
=\sum_{n}h^0(M,\Theta^r\otimes\pi^{*}\mathcal{O}_{\ls}(n))t^n\nonumber
\\
&=&\sum_{n}h^0(\ls, \pi_{*}(\Theta^r)\otimes\mathcal{O}_{\ls}(n))t^n\nonumber\\
&=&\large{\frac{1+3t^{2}+\sum_{i=3}^r ((i+1)t^{i}+(i-2)t^{i+1})}{(1-t)^{l+1}}}.\nonumber
\end{eqnarray}
When $r=2,$  it matches G\"ottsche's computation.   Hence we have
\begin{coro}Let $X$ and $L$ be as in Example \ref{exgt}.  Let $u=(0,L,\chi(u)=0)$ and $c_n=(2,0,n)$,  then we have for all $n\geq0$
\[\chi(M(c_n),\lu)=\chi(M(u),\lambda_{c_n})=h^0(M(u),\lambda_{c_n}).\]
\end{coro}

On $M$ we have an exact sequence for $r>0$
\begin{equation}\label{mexa}
0\ra\z^r\ra\z^{r+1}\ra D_{\z}(\z^{r+1})\ra0.
\end{equation}
Push it forward via $\pi$ to $\ls$.  By Statement $1$ in Theorem \ref{main},   we have
\begin{equation}\label{lsexa}
0\ra\pi_{*}\z^r\ra\pi_{*}\z^{r+1}\ra \pi_{*}D_{\z}(\z^{r+1})\ra0.
\end{equation}
Then we see that Statement $2$ in Theorem \ref{gtmrt} is just a consequence of the following proposition.
\begin{prop}\label{pfdz}For $r\geq2,$  $\pi_{*}D_{\z}(\z^r)=\mo_{\ls}(-r)^{\oplus r+1}\oplus\mo_{\ls}(-r-1)^{\oplus r-2}.$
\end{prop} 
Before proving this proposition,  we need to show some lemmas.
\begin{lemma}\label{gtcm}$D_{\z}$ is Cohen-Macaulay.
\end{lemma}
\begin{proof}This is because $M$ is Cohen-Macaulay and $D_{\z}$ is a divisor in $M.$
\end{proof}
\begin{lemma}\label{gtcod}$D_{\z}-D_{\z}^s$ is of codimension $\geq2$ in $D_{\z}.$
\end{lemma}
\begin{proof}Let $\mf$ be a strictly semi-stable sheaf S-equivalent to $\oplus_{i} \mf_i.$  Then $\mf$ has a nonzero global section if and only if one of the $\mf_i$'s has.  Hence strictly semi-stable points with nonzero global sections form a closed subscheme of codimension $1$ in $M-M^s.$  Then the lemma follows because $M-M^s$ is of codimension $\geq2$ in $M$.
\end{proof}
For the stable points in $D_{\z},$  we have the following description.
\begin{lemma}\label{dszd}Let $C$ be a curve in $\ls.$  Let $\mf$ be a sheaf of Euler characteristic zero with schematic support $C.$  Then $\mf$ is stable and has a nonzero global section,  i.e. $[\mf]\in D^s_{\z}$ $\Leftrightarrow$ $\mf$ lies in a non-splitting exact sequence
\begin{equation}\label{zetags}0\ra\mo_{C}\ra\mf\ra\mo_p\ra0,\end{equation}
with $p$ a point (with reduced structure) in $C,$  and $\mf$ does not contain a subsheaf of Euler characteristic zero.

Moreover,  if $\mf$ lies in the non-splitting sequence (\ref{zetags}) and contains a subsheaf of Euler characteristic zero,  then it is strictly semi-stable and $C$ is not integral.
\end{lemma}
\begin{proof}"$\Rightarrow$":
Let $\mf$ be a stable sheaf supported at $C$ with a nonzero global section,  then we have 
\begin{equation}\label{gtlef}\xymatrix{\mo_C\ar[r]^{s}&\mf,}
\end{equation}
with $s\neq0.$

If $s$ is not injective,  then its image is a quotient of $\mo_{C}$ hence is isomorphic to $\mo_{C'}$ with $C'$ some closed subscheme in $C.$  Since $s$ is nonzero,  $C'\neq\emptyset;$  and also $s$ is not injective,  $C'\subsetneq C.$  On the other hand,   $C$ is a curve of arithmetic genus $2$ and for any subscheme $C'\subsetneq C,$  $C'$ is of genus no larger than $1.$  Hence $\chi(\mo_{C'})\geq0$ which contradicts the stability of $\mf.$  Therefore we have that $s$ is injective.

As $s$ is injective,  it is easy to see that the cokernel of $s$ is a sheaf of dimension zero and of Euler characteristic $1,$  hence it is the structure sheaf over a reduced point $p$ with $p\in C.$  Hence we have  
\begin{equation}\label{gtrig}0\ra \mo_{C}\ra\mf\ra\mo_{p}\ra0.
\end{equation}
And of course $\mf$ can not have a subsheaf of Euler characteristic zero.  

"$\Leftarrow$":  Let $\mf$ be a sheaf which is a nontrivial extension of $\mo_{p}$ by $\mo_{C}$,  with $p$ a point with reduced structure in $C.$  Then we have 
\begin{equation}\label{gtexte}\xymatrix{0\ar[r]&\mo_C\ar[r]^{s}&\mf\ar[r]^{v}&\mo_{p}\ar[r]&0.}\end{equation}
It is easy to compute that $\chi(\mf)=0$ and also easy to see that $\mf$ is pure.  We want to show that for any proper subsheaf $\mg$ of $\mf,$  $\chi(\mg)$ is non-positive,  and this will suffice for the proof of the rest of the lemma.  

Given $\mg$ a proper subsheaf of $\mf,$  we see that $v(\mg)$ is a subsheaf of $\mo_{p}$,  hence $v(\mg)=0$ or $v(\mg)=\mo_{p}$.  

If $v(\mg)=0,$  then $\mg$ is actually a subsheaf of $\mo_C$ and hence it is a ideal sheaf $I$ of some closed subscheme $C'\subsetneq C.$  $C'$ can be of dimension $1$ or dimension $0.$  But in both cases we have $\chi(\mo_{C'})\geq0$ and hence $\chi(\mg)\leq-1$ because $\chi(\mg)+\chi(\mo_{C'})=\chi(\mo_C)=-1$.  

If $v(\mg)=\mo_p,$  then we have 
\[0\ra Ker\ra\mg\ra \mo_p\ra0.\]
$Ker$ is a subsheaf of $\mo_C$,   hence $\chi(Ker)\leq-1$ or $Ker=0.$  If $Ker=0,$  then $\mg\simeq\mo_p$ and the sequence (\ref{gtexte}) splits which is a contradiction.  If $Ker\neq0,$  then $\chi(\mg)\leq 0$ with equality if and only if $\chi(Ker)=-1.$  

And finally if $\chi(Ker)=-1$ and $Ker\neq\mo_C,$  then $C$ must be either reducible or non-reduced since $\mo_C/Ker$ can not be a sheaf of dimension $0$. 
\end{proof}
\begin{rem}\label{idss}
For any curve $C$ in $\ls,$  let $p$ be a point in $C$ with reduced structure and denote $I_p$ to be the ideal sheaf of $p$ in $C.$  Then from Lemma \ref{dszd} we see that $\mathcal{H}om(I_p,\mo_C)$ is semi-stable and has nonzero global sections for any $p\in C$.  And moreover when $p$ is a smooth point,  $\mathcal{H}om(I_p,\mo_C)$ is a line bundle on $C$.
\end{rem}

\begin{rem}\label{idssu}Actually for any single point $p\in C,$  (Ext$^1(\mo_{p},\mo_C)-\{0\})/\mathbb{G}_m$ is just one point and hence if $\mf$ lies in an exact sequence (\ref{gtexte}),  then $\mf\simeq \mathcal{H}om(I_p,\mo_C).$ 

Let $d_C$ be the dimension of $D_{\z}^s$ restricted to the fiber of $\pi$ over the curve $C.$  Now we know that $d_C$ is no larger than the dimension of the curve,  hence $d_C\leq1$ for every $[C]\in\ls.$  And when $C$ is integral,  $d_C=1.$  
\end{rem}
\begin{proof}[Proof of Statement $1$ in Theorem \ref{gtmrt}]We know by Remark \ref{idssu} that $D_{\z}^s$ restricted to every fiber is of dimension no larger than $1.$  $D_{\z}-D_{\z}^s$ restricted to a fiber over a non-integral curve $[C]$ is a finite set of points,  each of which corresponds to the S-equivalence classes of $\mo_{C''}\oplus\mo_{C'}(-1)$ with $C''$ a component of arithmetic genus $1$ and $C'\simeq \pone.$  Thus we know that $D_{\z}$ restricted to every fiber is of dimension no larger than $1$.  And by Lemma \ref{zdd} we know that $D_{\z}$ restricted to every fiber is of dimension $1$ and every fiber of $\pi$ is of dimension $2,$  and hence the statement.
\end{proof}

Denote $\mc$ to be the universal family of curves in $X\times\ls$ parametrized by $\ls.$  $\mc$ is a smooth projective scheme.
We then define a morphism from $\mc$ to $M$ with its image in $D_{\z}$ as following.

Let $\triangle:\mc\ra\mc\times_{\ls}\mc$ be the diagonal embedding.  Then $\triangle$ is a closed embedding and the image of $\triangle$ is a divisor in $\mc\times_{\ls}\mc.$
We denote $I_{\triangle}$ to be the ideal sheaf of $\triangle(\mc).$  Notice that $I_{\triangle}$ is not locally free on $\mc\times_{\ls}\mc$,  because $\mc\times_{\ls}\mc$ is not smooth.

Since $X\times\mc=X\times\ls\times_{\ls}\mc,$  we have
\begin{equation}\label{gtbig}\xymatrix{\mc\ar[r]^{\triangle}&\mc\times_{\ls}\mc\ar[r]^{i\times id_{\mc}}\ar[d]^{\overline{p_1}}&X\times\mc\ar[r]^{p_2}\ar[d]^{p_1}&\mc\ar[d]^{\overline{\pi}}\\&\mc\ar[r]^i&X\times\ls\ar[r]^p&\ls.}\end{equation}

We can see that $(i\times id_{\mc})_{*}\mathcal{H}om(I_{\triangle},\mo_{\mc\times_{\ls}\mc})$ is flat over $\mc$,  because restricted to the fiber over any point $p\in\mc,$  it is $\mathcal{H}om_{\mo_{C_p}}(I_p,\mo_{C_p})$ and has the same Hilbert polynomial restricted to every fiber.  And because of Remark \ref{idss} we know that $(i\times id_{\mc})_{*}\mathcal{H}om(I_{\triangle},\mo_{\mc\times_{\ls}\mc})$ is a flat family of semi-stable sheaves over $\mc$.   Then it induces a morphism $f:\mc\ra M$.  It is easy to see that its image is contained in $D_{\z}.$

We have a commutative diagram
\begin{equation}\label{gtcomm}\xymatrix{\mc\ar[r]^f\ar[dr]^{\overline{\pi}}&D_{\z}\ar[d]^{\pi}\\ &\ls}.\end{equation}

Notice that $\overline{\pi}_{*}(f^{*}\z^r)\simeq \pi_{*}f_{*}(f^{*}\z^r)\simeq \pi_{*}(f_{*}\mo_{\mc}\otimes\z^r).$  Proposition \ref{pfdz} follows immediately from the two following lemmas.
\begin{lemma}\label{gtfir}$f_{*}\mo_{\mc}\simeq\mo_{D_{\z}}.$
\end{lemma}  
\begin{lemma}\label{gtsec}$\overline{\pi}_{*}(f^{*}\z^r)\simeq\mo_{\ls}(-r)^{\oplus r+1}\oplus\mo_{\ls}(-r-1)^{\oplus r-2}.$
\end{lemma}
Before proving these two lemmas,  let us first give some notations.  Let $\ls^1$ be the open subscheme of $\ls$ containing integral curves.  We can see that $\ls-\ls^1$ is of codimension $\geq 2$ in $\ls.$  Denote $\mc^1$ (resp. $D_{\z}^{1}$) to be the preimage of $\ls^{1}$ along $\overline{\pi}$ (resp. $\pi$).  And $D_{\z}^o=D_{\z}^{1}\cap M^{sm},$  and also $\mc^o$ is the preimage of $D_{\z}^o$ along $f.$  Hence we have the following Cartesian diagram.
\begin{equation}\label{gtcar}\xymatrix{\mc^o\ar[r]\ar[d]^{f^o}&\mc^{1}\ar[r]\ar[d]^{f^{1}}&\mc\ar[d]^f\\ D_{\z}^o\ar[r]&D_{\z}^{1}\ar[r]&D_{\z}}.\end{equation}

 Notice that $D_{\z}-D_{\z}^{1}$ is of codimension $\geq2$ in $D_{\z}$.  This is because of Lemma \ref{gtcod} and Remark \ref{idssu}.  We also have
 \begin{lemma}\label{dzin}
 $D_{\z}$ is integral and $D_{\z}-D_{\z}^o$ is of codimension $\geq2$ in $D_{\z}.$  Hence $D_{\z}^o$ is dense in $D_{\z}$.
 \end{lemma} 
 \begin{proof}The moduli space $M$ is irreducible by Corollary \ref{irred},  and $D_{\z}$ is a divisor in $M.$  So if $D_{\z}$ is not integral,  then we can write $D_{\z}=D_1+D_2$ with $D_i$'s divisors in $M$ and $dim~D_i=dim~M-1=l+1.$  On the other hand,  $D_{\z}$ restricted to every fiber is of dimension $1.$  As a result for $i=1,2,$  the image of $D_i$ along $\pi$ is a closed subscheme of $\ls$ of dimension $l$ and hence is $\ls$.  Then we know that $D_{\z}$ restricted to a generic fiber of $\pi$ is not integral.  But this contradicts the fact that $D_{\z}$ restricted to a fiber over smooth curve is integral.  So we know that $D_{\z}$ is irreducible and hence any open set in $D_{\z}$ is dense in $D_{\z}$.  
 
Since $D_{\z}-D_{\z}^{1}$ is of codimension $\geq2$ in $D_{\z},$  to prove that $D_{\z}-D_{\z}^o$ is of codimension $\geq2,$  it is enough to prove that $D_{\z}^{1}-D_{\z}^o$ is of codimension $\geq2$ in $D_{\z}^{1}.$
Let $\mf$ be a sheaf in $D_{\z}^{1}-D_{\z}^o.$  Let $C$ be its supporting curve.  
 
Since $C$ is integral,  then $\mf\simeq\mathcal{H}om(I_p,\mo_C)$ with $p$ a singular point in $C.$  Hence $D_{\z}^{1}-D_{\z}^o$ restricted to the fiber over $C$ is empty if $C$ is smooth and contains finitely many points if $C$ is not smooth.  
  
 This finishes the proof of the lemma.
 \end{proof} 
 
\begin{proof}[Proof of Lemma \ref{gtfir}] $f$ is a projective morphism,  and by Lemma \ref{dzin} we know that $D_{\z}$ is integral.  Hence it is enough to show the following two statements:

$(1)$ $f$ is a birational map;

$(2)$  $D_{\z}$ is normal.

Because of Lemma \ref{gtcm} we know that $D_{\z}$ is Cohen-Macaulay.  Hence it is normal if and only if it is regular in codimension one.  Moreover since $D_{\z}-D_{\z}^o$ is of codimension $\geq2$,  it is enough to show that $D_{\z}^o$ is normal.
Hence both Statement $(1)$ and Statement $(2)$ will follow if we show that $f^o$ in (\ref{gtcar}) is an isomorphism. 

Now we focus on $f^o.$  Since (\ref{gtcar}) is a Cartesian diagram,  $f^o$ is projective.  And it is easy to see that $f^o$ is bijective,  and hence it is affine.  Then we also have $f^{\#}:\mo_{D_{\z}^o}\ra f^o_{*}\mo_{\mc^o}$ is injective because of the surjectivity of $f^o.$  Moreover $f^o$ will be an isomorphism if $f^{\#}$ is surjective.

To prove the surjectivity of $f^{\#},$  by Nakayama's lemma,  it is equivalent to show that $f^{\#}$ restricted to every fiber of $\pi$ is surjective,  i.e. for any $y\in\ls$ we have $f^{\#}_y:\mo_{D_{\z}^o}\otimes k(y)\ra f^o_{*}\mo_{\mc^o}\otimes k(y)$ is surjective.

We restrict the commutative diagram (\ref{gtcomm}) to $D_{\z}^o$ and get
\begin{equation}\label{gtcommrs}\xymatrix{\mc^{o}\ar[r]^{f^{o}}\ar[dr]^{\overline{\pi}^o}&D_{\z}^{o}\ar[d]^{\pi^o}\\ &\ls^{1}}.\end{equation}

Notice that both $\overline{\pi}^o$ and $\pi^o$ are flat.  This is because $\ls^{1}$ is regular,  both $\mc^o$ and $D_{\z}^o$ are Cohen-Macaulay and integral and also every fiber of $\overline{\pi}^o$ and $\pi^o$ is of dimension $1$.  

Since $f^o$ is bijective,  we have $R^if^o_{*}\mo_{\mc^o}=0$ for all $i>0.$  And hence $f^o_{*}$ commutes with the restriction to the fiber,  i.e. $f_{*}^o\mo_{\mc^o}\otimes k(y)\simeq f_{*}^o(\mo_{\mc^o}\otimes k(y))$ for any $y\in\ls.$  Hence to prove the surjectivity of $f^{\#}_y$ it is enough to show that $f^o$ restricted to the fiber over $y$ is an isomorphism.

Let $C$ be the curve corresponding to the point $y$ in $\ls.$  Denote $C_y:=\mc^o\times Spec~k(y)$ and $D_y:=D_{\z}^o\times Spec~k(y).$  One then can see that $D_y$ is the moduli space parametrizing line bundles on $C$ which are of degree $1$ and have nonzero global sections,  and hence there is a morphism $h:D_y\ra Pic^{1}~C$.  Now we view the smooth locus of $C$ as a closed subscheme of $Pic^1~C$ by assigning every smooth point $p$ to $[\mo_{C}(p)].$  It is easy to see that the image of $D_y$ via $h$ is $C_y$ and $h$ provides an inverse of $f^o_y.$  Hence $f^o$ restricted to every fiber is an isomorphism and hence the lemma.
\end{proof}

\begin{proof}[Proof of Lemma \ref{gtsec}]We define $\mc^{s}$ to be the open subscheme in $\mc$ by excluding all singular points on every fiber of $\overline{\pi}$ in diagram (\ref{gtbig}).  One sees that $\mc-\mc^s$ is of codimension $\geq2$ in $\mc.$  Denote $j:\mc^s\ra\mc$ be the open embedding.  Because $\mc$ is smooth,  we have that 
\begin{equation}\label{rsop}f^{*}\z^r\simeq j_{*}j^{*}f^{*}\z^r.\end{equation}
We first compute $j^{*}f^{*}\z^r$ and then push it forward along $j$ to get $f^{*}\z^r.$  
We have a Cartesian diagram  
\begin{equation}\label{gtopen}\xymatrix{\mc^{s}\ar[r]^{\triangle^s}\ar[d]^j&\mc\times_{\ls}\mc^{s}\ar[r]^{i\times id_{\mc^s}}\ar[d]^{id_{\mc}\times j}&X\times\mc^{s}\ar[r]^{p^s_2}\ar[d]^{id_{X\times\ls}\times j}&\mc^{s}\ar[d]^j\\ \mc\ar[r]^{\triangle}&\mc\times_{\ls}\mc\ar[r]^{i\times id_{\mc}}\ar[d]^{\overline{p_1}}&X\times\mc\ar[r]^{p_2}\ar[d]^{p_1}&\mc\ar[d]^{\overline{\pi}}\\&\mc\ar[r]^i&X\times\ls\ar[r]^p&\ls.}\end{equation}

By the universal property of $\z,$  we have $(j^{*}f^{*}\z)^{\vee}\simeq det~R^{\bullet}p^s_{2}\circ(id_{X\times\ls}\times j)^{*}(i\times id_{\mc})_{*}\mathcal{H}om(I_{\triangle},\mo_{\mc\times_{\ls}\mc})\simeq  det~R^{\bullet}p^s_{2}\circ(i\times id_{\mc^s})_{*}(id_{\mc}\times j)^{*}\mathcal{H}om(I_{\triangle},\mo_{\mc\times_{\ls}\mc}).$  And let $I^s_{\triangle}$ denote $I_{\triangle}$ restricted to $\mc\times_{\ls}\mc^s,$  then $(id_{\mc}\times j)^{*}\mathcal{H}om(I_{\triangle},\mo_{\mc\times_{\ls}\mc})=\mathcal{H}om(I_{\triangle}^s,\mo_{\mc\times_{\ls}\mc^s}).$

Notice that $\overline{\pi}\circ j$ is smooth and hence $\overline{p_1}\circ (id_{\mc}\times j)$ is smooth.  Then because $\mc$ is smooth,  $\mc\times_{\ls}\mc^s$ is smooth.  Then $I^s_{\triangle}$ is locally free on $\mc\times_{\ls}\mc^s$,  and so is $\mathcal{H}om(I^s_{\triangle},\mo_{\mc\times_{\ls}\mc^s})$.  We denote $I_{\triangle}^{s\vee}$ to be $\mathcal{H}om(I^s_{\triangle},\mo_{\mc\times_{\ls}\mc^s}).$  Since $\triangle^s(\mc^s)=\triangle(\mc)\cap(\mc\times_{\ls}\mc^s)$,  we have an exact sequence on $\mc\times_{\ls}\mc^s$ 
\begin{equation}\label{firdual}0\ra\mo_{\mc\times_{\ls}\mc^s}\ra I_{\triangle} ^{s\vee}\ra\mo_{\triangle^s(\mc^s)}\otimes I_{\triangle} ^{s\vee}\ra0.
\end{equation}

We know that 
\begin{equation}\label{gtorin}(j^{*}f^{*}\z)^{\vee}\simeq det~R^{\bullet}p^s_{2}\circ (i\times id_{\mc^s})_{*}I_{\triangle} ^{s\vee}.\end{equation}  

On the other hand,  $i$ is a closed embedding and so is $i\times id_{\mc}.$  Hence 
\begin{equation}\label{gtch}R^{\bullet}p^s_{2}\circ (i\times id_{\mc^s})_{*}I_{\triangle} ^{s\vee}\simeq R^{\bullet}(p^s_{2}\circ (i\times id_{\mc^s}))I_{\triangle} ^{s\vee}.\end{equation}

And because of sequence (\ref{firdual}),  we have 
\begin{eqnarray}\label{gttim}&& det~R^{\bullet}(p^s_{2}\circ (i\times id_{\mc^s}))I_{\triangle} ^{s\vee}\\ &\simeq& (det~R^{\bullet}(p^s_{2}\circ (i\times id_{\mc^s}))\mo_{\mc\times_{\ls}\mc^s})\otimes (det~R^{\bullet}(p^s_{2}\circ (i\times id_{\mc^s}))\mo_{\triangle^s(\mc^s)}\otimes I_{\triangle} ^{s\vee}).\nonumber\end{eqnarray}

First we compute $det~R^{\bullet}(p^s_{2}\circ (i\times id_{\mc^s}))\mo_{\mc\times_{\ls}\mc^s}.$

Since the diagram (\ref{gtopen}) is Cartesian and $\overline{\pi}\circ j$ is flat,  we have that 
\[[R^{\bullet}(p^s_{2}\circ (i\times id_{\mc^s}))\mo_{\mc\times_{\ls}\mc^s}]=[(\overline{\pi}\circ j)^{*} R^{\bullet}(p\circ i) \mo_{\mc}];\]
and
\begin{equation}\label{gtexch}det~R^{\bullet}(p^s_{2}\circ (i\times id_{\mc^s}))\mo_{\mc\times_{\ls}\mc^s}\simeq(\overline{\pi}\circ j)^{*} det~R^{\bullet}(p\circ i) \mo_{\mc}\end{equation}
Using the exact sequence on $X\times\ls$
\begin{equation}\label{dfuc}0\rightarrow q^{*}\mathcal{O}_{X}(-L)\otimes p^{*}\mathcal{O}_{\ls}(-1)\rightarrow\mathcal{O}_{X\times\ls}\rightarrow\mathcal{O}_{\mathcal{C}}\rightarrow0,\end{equation}
where $q:X\times\ls\ra X$ is the projection to the first factor,  we get that 
\begin{eqnarray}\label{gtpfsf}[det~R^{\bullet}(p\circ i) \mo_{\mc}]&=&(det~R^{\bullet}p \mo_{X\times\ls})\otimes (det~R^{\bullet}p(q^{*}\mathcal{O}_{X}(-L)\otimes p^{*}\mathcal{O}_{\ls}(-1)))^{\vee}\nonumber\\ &=& \mo_{\ls}(1)^{\otimes \chi(\mo_X(-L))}=\mo_{\ls}(2).\end{eqnarray}  
The last equality is because $\chi(\mo_X(-L))=\chi(\mo_X)+g_L-1=2.$

Because of (\ref{gtexch}) and (\ref{gtpfsf}),  we have 
\begin{equation}\label{gtor}det~R^{\bullet}(p^s_{2}\circ (i\times id_{\mc^s}))\mo_{\mc\times_{\ls}\mc^s}\simeq j^{*}\overline{\pi}^{*}\mo_{\ls}(2).\end{equation}

Now we compute $det~R^{\bullet}(p^s_{2}\circ (i\times id_{\mc^s}))\mo_{\triangle^s(\mc^s)}\otimes I_{\triangle} ^{s\vee}.$

Notice that $p_2^s\circ (i\times id_{\mc^s})$ restricted on $\triangle^s(\mc^s)$ is an isomorphism and $p_2^s\circ (i\times id_{\mc^s})\circ\triangle^s=id_{\mc^s}.$  Hence
\begin{equation}\label{rtmc}R^{\bullet}(p^s_{2}\circ (i\times id_{\mc^s}))\mo_{\triangle^s(\mc^s)}\otimes I_{\triangle} ^{s\vee}\simeq (\triangle^{s})^{*} I^{s\vee}_{\triangle}.
\end{equation}

And morover $(\triangle^{s})^{*} I^{s\vee}_{\triangle}$ is the relative tangent bundle $\mathcal{T}_{\mc^s/\ls}$ of the smooth morphism $\overline{\pi}\circ j:\mc^s\ra\ls.$  On $X\times\ls$ we have 
\begin{equation}\label{notex}0\ra\mathcal{T}_{\mc/\ls}\ra i^{*}\mathcal{T}_{X\times\ls/\ls}\ra\mathcal{N}_{\mc},\end{equation}
where $i$ is the closed embedding of $\mc$ into $X\times\ls$ as in diagram (\ref{gtopen}) and $\mathcal{N}_{\mc}$ is the normal bundle on $\mc.$  Hence $\mathcal{N}_{\mc}\simeq i^{*}(q^{*}\mo_X(L)\otimes p^{*}\mo_{\ls}(1)).$  When we restrict the sequence (\ref{notex}) to $\mc^s,$  it becomes a short exact sequence and hence we get
\begin{equation}\label{gtex}0\ra\mathcal{T}_{\mc^s/\ls}\ra j^{*}i^{*}\mathcal{T}_{X\times\ls/\ls}\ra j^{*}i^{*}(q^{*}\mo_X(L)\otimes p^{*}\mo_{\ls}(1))\ra0,\end{equation}

Because of (\ref{rtmc}) and (\ref{gtex}) we know that \begin{eqnarray}\label{gtscr}&&det~R^{\bullet}(p^s_{2}\circ (i\times id_{\mc^s}))\mo_{\triangle^s(\mc^s)}\otimes I_{\triangle} ^{s\vee}\simeq det~\mathcal{T}_{\mc^s/\ls}\nonumber\\ &\simeq& (det~j^{*}i^{*}\mathcal{T}_{X\times\ls/\ls})\otimes (det~j^{*}i^{*}(q^{*}\mo_X(L)\otimes p^{*}\mo_{\ls}(1)))^{\vee}.\end{eqnarray}

Since $\mathcal{T}_{X\times\ls/\ls}\simeq q^{*}\mathcal{T}_X,$  we have 
\begin{equation}\label{gttr}det~R^{\bullet}(p^s_{2}\circ (i\times id_{\mc^s}))\mo_{\triangle^s(\mc^s)}\otimes I_{\triangle} ^{s\vee}\simeq det~\mathcal{T}_{\mc^s/\ls}\simeq j^{*}i^{*}(q^{*}\mo_{X}(-L-K)\otimes p^{*}\mo_{\ls}(-1)).
\end{equation}

Combining (\ref{gtorin}) (\ref{gtch}) (\ref{gttim}) (\ref{gtor}) and (\ref{gttr}),  we finally have 
\[j^{*}f^{*}\z\simeq j^{*}i^{*}(q^{*}\mo_{X}(L+K)\otimes p^{*}\mo_{\ls}(-1)).\]
And moreover because of (\ref{rsop}),  we have $f^{*}\z\simeq i^{*}(q^{*}\mo_{X}(L+K)\otimes p^{*}\mo_{\ls}(-1)).$

Now in order to compute $\overline{\pi}_{*}f^{*}\z^r,$  we tensor the sequence (\ref{dfuc}) by $q^{*}\mo_{X}(r(L+K))\otimes p^{*}\mo_{\ls}(-r)$ and get
\begin{equation}\label{dfuctr}0\rightarrow q^{*}\mathcal{O}_{X}(r(L+K)-L)\otimes p^{*}\mathcal{O}_{\ls}(-r-1)\rightarrow q^{*}\mo_{X}(r(L+K))\otimes p^{*}\mo_{\ls}(-r)\rightarrow f^{*}\z^r\rightarrow0,
\end{equation}

We have $f^{*}\z^r$ on the right in the sequence (\ref{dfuctr}) because 
\[\mathcal{O}_{\mathcal{C}}\otimes (q^{*}\mo_{X}(r(L+K))\otimes p^{*}\mo_{\ls}(-r))\simeq i^{*}(q^{*}\mo_{X}(r(L+K))\otimes p^{*}\mo_{\ls}(-r))\simeq f^{*}\z^r.\]

 As $X=\mathbb{P}(\mo_{\pone}\oplus\mo_{\pone}(-e))$ and $L=2G+(e+3)F,$  $K=-2G-2F$ with $F$ the fiber class and $G.G=-e,$  we have that for all $r\geq2,$   

$(1)$ $H^0(r(L+K)-L)=0,$  $H^2(r(L+K)-L)=H^0(-(r-1)(L+K))^{\vee}=0,$  and hence $h^1(r(L+K)-L)=-\chi(r(L+K)-L)=r-2;$

$(2)$ $H^i(r(L+K))=0,$  for $i>0,$  and $h^0(r(L+K))=\chi(r(L+K))=r+1.$

Then using sequence (\ref{dfuctr}) one can easily compute $\overline{\pi}_{*}f^{*}\z^r$ and get the expected result.   And this finishes the proof of the lemma.
\end{proof}

\section*{Appendix.}
\subsection*{A.}
The conclusion of following lemma is somehow well-known,  but we still give a proof here because we didn't find any good reference for any proof.

\begin{lemma}\label{llfr}Let $X$ be a smooth projective surface,  and let $\omega$ be its dualizing sheaf.  Let $F$ be a pure sheaf of dimension one on $X.$  Then $F$ has a locally free resolution of length one and $F\simeq F^{DD}:=\mathcal{E}xt^1(\mathcal{E}xt^1(F,\omega),\omega).$
\end{lemma}
\begin{proof}On $X$ we have an exact sequence
\begin{equation}\label{lfr}
0\rightarrow E\rightarrow P \rightarrow F\rightarrow 0,
\end{equation}
where $P$ is a locally free sheaf on $X.$  To prove the first statement of the lemma,  it is enough to show $E\simeq E^{DD}:=\mathcal{H}\emph{om}(\mathcal{H}\emph{om}(E,\omega),\omega)$,  i.e. $E$ is reflexive.

Let $\mathcal{H}\emph{om}(-,\omega)$ act on (\ref{lfr}) and we get 
\begin{equation}\label{extact}
0\rightarrow \mathcal{H}\emph{om}(P,\omega)\rightarrow\mathcal{H}\emph{om}(E,\omega)\rightarrow\mathcal{E}\emph{xt}^1(F,\omega)\rightarrow0
\end{equation}
The $0$ on the left hand side is because of the vanishing of $\mathcal{H}\emph{om}(F,\omega),$  which can be deduced from the fact that $F$ is a torsion sheaf and $\omega$ is torsion free.  The $0$ on the right hand side is because of the vanishing of $\mathcal{E}\emph{xt}^1(P,\omega),$  which can be deduced from the fact that $P$ is locally free.  Moreover $\mathcal{H}\emph{om}(P,\omega)$ is locally free and $\mathcal{E}\emph{xt}^1(F,\omega)$ is a torsion sheaf of dimension one.

Let $\mathcal{H}\emph{om}(-,\omega)$ act on (\ref{extact}) and then we get
\begin{equation}\label{exttact}
0\rightarrow \mathcal{H}\emph{om}(\mathcal{E}\emph{xt}^1(F,\omega),\omega)\rightarrow E^{DD}\rightarrow P^{DD}\rightarrow F^{DD}\rightarrow \mathcal{E}\emph{xt}^1(\mathcal{H}\emph{om}(E,\omega),\omega)\rightarrow0.
\end{equation}  

$\mathcal{H}\emph{om}(\mathcal{E}\emph{xt}^1(F,\omega),\omega)=0,$  because $\mathcal{E}\emph{xt}^1(F,\omega)$ is a torsion sheaf and $\omega$ is torsion free.  So there is a injective morphism $E^{DD}\rightarrow P^{DD},$  sending $E^{DD}$ as a subsheaf of $P^{DD}\simeq P.$    $E^{DD}/E$ is a subsheaf of $F$ by sequence (\ref{lfr}).  But $F$ is pure and $E^{DD}/E$ is of dimension zero.  Thus $E^{DD}/E$ is zero and $E\simeq E^{DD}$.   

As $E$ is reflexive hence locally free,  $\mathcal{E}\emph{xt}^1(\mathcal{H}\emph{om}(E,\omega),\omega)=0$ and $\mathcal{H}\emph{om}(E,\omega)$ is locally free.  Then sequence (\ref{exttact}) can be rewritten as 
\begin{equation}
0\rightarrow E^{DD}\rightarrow P^{DD}\rightarrow F^{DD}\rightarrow 0
\end{equation}
Since every pure sheaf can be embedded into its reflexive hull,  we have the commutative diagram  
\begin{equation}
\xymatrix{
0\ar[r]&E\ar[r]\ar[d]^{\theta_E} &P\ar[r] \ar[d]^{\theta_P}&F\ar[r]\ar[d]^{\theta_F} &0\\
0\ar[r]&E^{DD}\ar[r]&P^{DD}\ar[r]&F^{DD}\ar[r]&0}
\end{equation} 
$\theta_E$ and $\theta_P$ are both isomorphisms,  so is $\theta_F$,  and hence $F\simeq F^{DD}.$
\end{proof} 

\subsection*{B.}
In this subsection we want to prove the following lemma.
\begin{lemma}\label{finitegz}Let $X$ be any Hirzebruch surface,  with $F$ the fiber class and $G$ the section class.  Let $\mk$ be a semistable sheaf of class $(0,nF,0)$ on $X$.  If we fixed the schematic support of $\mk$ in $|nF|$,  then there are finitely many isomorphic classes of such $\mk$.\end{lemma}
\begin{proof}With no loss of generality,  we assume that the schematic support of $\mk$ is connected hence equals to $nC$ with $C\simeq \pone$ since $F.F=0.$  We then want to show that $\mk\simeq\bigoplus_{i}\mo_{n_iC}(-G)$,  with $n_i$ positive integers such that $\sum_in_i=n$. 

We use induction.  When $n=1$,  $\mk\simeq\mo_{C}(-1)\simeq\mo_{C}(-G)$ since $G.F=1$.  Then assume that we have proved the statement for all $n<n_0.$  Let $n=n_0.$  By Proposition \ref{isomls} $\mk$ is $S$-equivalent to $\mo_{C}(-1)^{\oplus n}$,  so it has $\mo_{C}(-1)$ as a quotient.  Hence we have the exact sequence
\begin{equation}\label{febuth}0\ra\mk'\ra\mk\ra\mo_{C}(-1)\ra0.\end{equation}
By assumption,  we have $\mk'\simeq\bigoplus_{i=1}^{N}\mo_{n'_iC}(-G)$ with $\sum_in'_i=n_0-1.$  We also assume $0<n'_1\leq n'_2\leq\ldots\leq n'_N$.  By a direct computation and Hirzebruch-Riemann-Roch,  we have 
\begin{eqnarray}dim~Ext^1(\mo_{C}(-1),\bigoplus_{i=1}^{N}\mo_{n'_iC}(-G))&=&dim~Hom(\mo_{C}(-1),\bigoplus_{i=1}^{N}\mo_{n'_iC}(-G))\nonumber \\ &=&\sum_{i=1}^N dim~Hom(\mo_{C}(-1),\mo_{n'_iC}(-G)).\nonumber\end{eqnarray}
For each $\mo_{n'_iC}(-G)$ we have the following exact sequence
\begin{equation}\label{dimhom}\xymatrix@C=0.6cm{0\ar[r]&\mo_{(n'_i-1)C}(-G)\ar[r]&\mo_{n'_iC}(-G)\ar[r]^r&\mo_{C}(-1)\ar[r]&0.}
\end{equation}
For every nonzero element $s\in$Hom$(\mo_{C}(-1),\mo_{n'_iC}(-G))$,  $r\circ s$ must be either zero or isomorphic.  If it was isomorphic,  then the sequence (\ref{dimhom}) would split.  Hence $r\circ s=0$ and Hom$(\mo_{C}(-1),\mo_{n'_iC}(-G))\simeq$Hom$(\mo_{C}(-1),\mo_{(n'_i-1)C}(-G))$.  So by induction we know that $dim~$Hom$(\mo_{C}(-1),\mo_{n'_iC}(-G))=1$ for all $n'_i>0.$ 

Ext$^1(\mo_{C}(-1),\bigoplus_{i=1}^{N}\mo_{n'_iC}(-G))\simeq \mathbb{C}^{N}$.  We then assign to every element $\underline{t}=(t_1,\ldots,t_N)\in\mathbb{C}^N-\{0\}$ a sequence as follow.  Let $0<i_0\leq N$ such that $t_{i_0}\neq0$ and $t_i=0$ for all $i>i_0.$
\begin{equation}\label{sqtn}\xymatrix@C=0.6cm{0\ar[r]&\bigoplus_{i=1}^{N}\mo_{n'_iC}(-G)\ar[r]^{f^{\underline{t}}~~~~~~~~~}&\bigoplus_{i\neq i_0}\mo_{n'_iC}(-G)\oplus\mo_{(n'_{i_0}+1)C}(-G)\ar[r]&\mo_{C}(-1)\ar[r]&0.}
\end{equation} 
The morphism $f^{\underline{t}}$ restricted on $\mo_{n'_iC}(-G)$ is an isomorphism to its image for $i\neq i_0.$  The image of $\mo_{n'_{i_0}C}(-G)$ via $f^{\underline{t}}$ is contained in $\bigoplus_{i\leq i_0}\mo_{n'_iC}(-G)\oplus\mo_{(n'_{i_0}+1)C}(-G).$  Let $f^{\underline{t}}|_{\mo_{n'_{i_0}C}(-G)}=(g_1,\ldots,g_{i_0}).$  Since $n'_i\leq n'_{i_0}$ for all $i<i_0$,  $n'_iC$ can be viewed as a subscheme of $n'_{i_0}C$.  For $i<i_0$ and $t_i\neq0$,  $g_i:\mo_{n'_{i_0}C}(-G)\ra\mo_{n'_{i}C}(-G)$ is the restriction (up to scalar) of $\mo_{n'_{i_0}C}(-G)$ to $n'_iC$.  $g_i=0$ if $t_i=0.$  Finally $g_{i_0}$ is the usual inclusion (up to scalar) of $\mo_{n'_{i_0}C}(-G)$ into $\mo_{(n'_{i_0}+1)C}(-G)$.

Hence we have that $\mk$ in (\ref{febuth}) must have the form $\bigoplus_{i}\mo_{n_{i}C}(-G)$  and thus the lemma.  
\end{proof}
\subsection*{C.}
In this subsection we let $X$ be a smooth complex projective surface.  We have the good quotient $\phi:\Omega\ra M(u)$ from the Quot-scheme $\Omega$ to the moduli space $M(u)$ of semistable sheaves of class $u=(0,L,\chi(u)=n)$ with $L$ some effective line bundle on $X$.   Moreover assume $L'.K<0,\forall 0<L'\leq L$ with $K$ the canonical divisor on $X.$  
Then there is a natural morphism $\pi:\Omega\ra\ls.$

Denote $\Omega^{sm}$ to be the open subscheme in $\Omega$ consisting of quotients that are locally free of rank $1$ on their supports.  We have a universal sheaf on $X\times\Omega^{sm}.$  We denote it as $\mathcal{E}$.  Then we have 
\begin{equation}\label{ppq}
\xymatrix{
  \mathcal{E}  \ar[r]
                & X\times \Omega^{sm} \ar[ld]^{q} \ar[d]^{p}  \\
                X & \Omega^{sm}  \ar[d]^{\pi^{sm}}         \\ 
                &\ls }
\end{equation}
$\pi^{sm}$ is smooth by Proposition 2.9 in \cite{lee}.
We have the following proposition.
\begin{prop}\label{qutang}On $\Omega^{sm},$  we have \[det~R^{\bullet}(p\circ\mathcal{H}om(\mathcal{E},\mathcal{E}))\simeq(\pi^{sm})^{*}\mo_{\ls}(-1)^{\otimes L.K}.\]
\end{prop}
\begin{proof}First notice that
\begin{equation}\label{stepone}det~[R^{\bullet}(p\circ \mathcal{H}om(\me,\me))]=det~[R^{\bullet}p\circ R^{\bullet}\mathcal{H}om(\me,\me)].\end{equation}

We have a Cartesian diagram
\begin{equation}\label{impt}
\xymatrix@C=0.5cm{\mathcal{C}^{\Omega}\ar[rr]^i\ar[d]_{\pi_{\mathcal{C}}}&&X\times \Omega^{sm}\ar[rr]^p\ar[d]&&\Omega^{sm}\ar[d]^{\pi^{sm}}\\
\mc\ar[rr]_{i_1}&&X \times\ls\ar[rr]_{p_1}&&\ls.}
\end{equation}
where $\mc$ is the universal family of curves in $\ls.$

We know that $\pi^{sm}$ is smooth and hence $\pi^{\mc}$ is smooth.  And $\mc$ is smooth in $X\times\ls$,  hence $\mc^{\Omega}$ is smooth in $X\times\Omega^{sm}.$  The universal sheaf $\me$ is supported at $\mc^{\Omega}$ and it is locally free on every fiber of $p\circ i.$  On the other hand,  since $p_1\circ i_1$ is flat,  $p\circ i$ is flat and hence $\me$ is locally free on $\mc^{\Omega}$.  Now let us view $\me$ as a locally free sheaf on $\mc^{\Omega}.$  
Since $i$ in (\ref{impt}) is a closed embedding,  for $\mf$ any coherent sheaf on $X\times\Omega^{sm},$  by coherent duality we have
\[[R^{\bullet}\mathcal{H}om_{X\times\Omega^{sm}}(\mathcal{F},i_{*}\mathcal{E})]=[i_{*}R^{\bullet}\mathcal{H}om_{\mc^{\Omega}}(\mo_{\mc^{\Omega}}\otimes^{L}\mathcal{F},\mathcal{E})],\]
where $\otimes^L$ means the flat tensor as $\mo_{X\times\Omega^{sm}}$-modules.

Since $\me$ is locally free on $\mc^{\Omega},$  we have in $\hk(\mc^{\Omega})$
\[[R^{\bullet}\mathcal{H}om_{\mc^{\Omega}}(\mo_{\mc^{\Omega}}\otimes^L\mathcal{E},\mathcal{E}))]=[R^{\bullet}\mathcal{H}om_{\mc^{\Omega}}(\mo_{\mc^{\Omega}}\otimes^L\mo_{\mc^{\Omega}},\mo_{\mc^{\Omega}}))].\]
And hence
\begin{eqnarray}\label{steptwo}[R^{\bullet}\mathcal{H}om_{X\times\Omega^{sm}}(i_{*}\mathcal{E},i_{*}\mathcal{E}))]&=&[i_{*}R^{\bullet}\mathcal{H}om_{\mc^{\Omega}}(\mo_{\mc^{\Omega}}\otimes^L\mo_{\mc^{\Omega}},\mo_{\mc^{\Omega}}))]\nonumber\\&=&[R^{\bullet}\mathcal{H}om_{X\times\Omega^{sm}}(i_{*}\mo_{\mc^{\Omega}},i_{*}\mo_{\mc^{\Omega}}))].\end{eqnarray}
We have exact sequence on $X\times\Omega^{sm}$
\[0\ra q^{*}\mo_{X}(-L)\otimes p^{*}(\pi^{sm})^{*}\mo_{\ls}(-1)\ra\mo_{X\times\Omega^{sm}}\ra\mo_{\mc^{\Omega}}\ra0.\]
Hence
\begin{eqnarray}&&[R^{\bullet}\mathcal{H}om_{X\times\Omega^{sm}}(i_{*}\mo_{\mc^{\Omega}},i_{*}\mo_{\mc^{\Omega}}))]\nonumber\\ &=&[\mathcal{H}om_{X\times\Omega^{sm}}(\mo_{X\times\Omega^{sm}},\mo_{X\times\Omega^{sm}})]\nonumber\\ &+&[\mathcal{H}om_{X\times\Omega^{sm}}(q^{*}\mo_{X}(-L)\otimes p^{*}(\pi^{sm})^{*}\mo_{\ls}(-1),q^{*}\mo_{X}(-L)\otimes p^{*}(\pi^{sm})^{*}\mo_{\ls}(-1))]\nonumber\\
&-&[\mathcal{H}om_{X\times\Omega^{sm}}(\mo_{X\times\Omega^{sm}},q^{*}\mo_{X}(-L)\otimes p^{*}(\pi^{sm})^{*}\mo_{\ls}(-1))]\nonumber\\ &-&[\mathcal{H}om_{X\times\Omega^{sm}}(q^{*}\mo_{X}(-L)\otimes p^{*}(\pi^{sm})^{*}\mo_{\ls}(-1),\mo_{X\times\Omega^{sm}})]\nonumber\end{eqnarray}

Hence we know that 
\begin{eqnarray}\label{stepthree}[R^{\bullet}\mathcal{H}om_{X\times\Omega^{sm}}(i_{*}\mo_{\mc^{\Omega}},i_{*}\mo_{\mc^{\Omega}}))]&=&2[\mo_{X\times\Omega^{sm}}]-[q^{*}\mo_{X}(-L)\otimes p^{*}(\pi^{sm})^{*}\mo_{\ls^{int}}(-1)]\nonumber\\ &-&[q^{*}\mo_{X}(L)\otimes p^{*}(\pi^{sm})^{*}\mo_{\ls^{int}}(1)]\end{eqnarray}

Put (\ref{stepone}) (\ref{steptwo}) (\ref{stepthree}) together,  we get
\begin{eqnarray}det~[R^{\bullet}(p\circ \mathcal{H}om(\me,\me))]&=&(\pi^{sm})^{*}\mo_{\ls^{int}}(-1)^{\otimes (\chi(L)-\chi(-L))}\nonumber\\ &=&(\pi^{sm})^{*}\mo_{\ls^{int}}(-1)^{\otimes L.K}\end{eqnarray}
This finishes the proof of the proposition.
\end{proof} 
\begin{flushleft}{\textbf{Acknowledgments.}}  I would like to thank Lothar G\"ottsche for his guidance and Barbara Fantechi,  Eduardo de Sequeira Esteves and Ramadas Ramakrishnan Trivandrum for many helpful discussions. 
\end{flushleft}


\begin{thebibliography}{99}
\bibitem{new}
Luis \'Alvarez-C\'onsul,  Alastair King,  \emph{A functorial construction of Moduli of Sheaves.}  arXiv.math/0602032v2.

\bibitem{mofir}
Alina Marian,  Dragos Oprea,\emph{A tour of theta dualities on moduli spaces of sheaves, Curves and abelian varieties,} 175-202, Contemporary Mathematics, 465, American Mathematical Society, Providence, Rhode Island (2008).

\bibitem{mosec}
Alina Marian,  Dragos Oprea, \emph{Generic strange duality for K3 surfaces},  arXiv:1005.0102v1 [math.AG].

\bibitem{bout}
J. F. Boutot,  \emph{Singularit\'es rationnelles et quotients par les groupes r\'eductifs.}  Invent.  Math.  88 (1987) p.65-68.

\bibitem{nila}
G. Danila,  \emph{R\'esultats sur la conjecture de dualit\'e \'etrange sur le plan projectif.} Bull. Soc. Math. France 130 (2002), 1--33. 

\bibitem{adv}
Geir Ellingsrud, Lothar G\"ottsche, Manfred Lehn,  \emph{On the cobordism class of the Hilbert scheme of a surface}, J. Algebraic Geom. 10 (2001), no. 1, 81--100.

\bibitem{gro}
A. Grothendieck.  \emph{Cohomologie locale des faisceaux coh\'erents et Th\'eor\`emes de Lefschetz locaux et globaux} (SGA 2) 

\bibitem{ha}
R. Harshorne,  \emph{Algebraic Geometry}.  GTM 52,  Springer Verlag,  New York(1977).

\bibitem{dan}
D. Huybrechts,  M. Lehn,  \emph{The Geometry of Moduli Spaces of Sheaves}.   Friedr.  Vieweg \& Sohn Verlagsgesellschaft mbH,  Braunschweig/Wiesbaden,  1997.

\bibitem{lee}
J. Le Potier,  \emph{Faisceaux Semi-stables de dimension $1$ sur le plan projectif}.  Rev. Roumaine Math.  Pures Appl.,  38(1993),7-8, 635-678.  

\bibitem{le}
J. Le Potier,  \emph{Faisceaux semi-stables et syst\`emes coh\'erents}. Proceedings de la Conference de Durham (July 1993),  Cambridge University Press (1995),  p.179-239

\bibitem{mum}
D. Mumford,  \emph{Lectures on curves on an algebraic surface.}  Annals Maths.  Studies 59,  Univ.  Press,  Princeton (1966).

\bibitem{git}
D. Mumford, \emph{Geometric Invariant Theory.}  Springer-Verlag Berlin-Heidelberg-New York (1965)

\end{thebibliography}
\end{document}